\documentclass[10pt]{amsart}
\usepackage[matrix,arrow,curve]{xy}
\usepackage{amssymb,amsmath,euscript}
\usepackage[cp1251]{inputenc}
\theoremstyle{definition}\newtheorem{theorem}{Theorem}[section]
\newtheorem*{theorem*}{theorem}
\newtheorem{proposition}[theorem]{Proposition}
\newtheorem{lemma}[theorem]{Lemma}
\newtheorem{corollary}[theorem]{Corollary}

\theoremstyle{definition}\newtheorem{definition}[theorem]{Definition}
\theoremstyle{remark}

\renewcommand{\to}[1][]{\xrightarrow{#1}}
\DeclareMathOperator{\Var}{\mathrm{Var}}\DeclareMathOperator{\U}{\mathrm{U}}\DeclareMathOperator{\Ann}{\mathrm{Ann}}\DeclareMathOperator{\Z}{\mathrm{Z}}\DeclareMathOperator{\cls}{\mathrm{cls}}\DeclareMathOperator{\I}{\mathrm{I}}\DeclareMathOperator{\rk}{\mathrm{rk}}\DeclareMathOperator{\Id}{\mathrm{Id}}\DeclareMathOperator{\End}{\mathrm{End}}\DeclareMathOperator{\pr}{\mathrm{pr}}\DeclareMathOperator{\Gr}{\mathrm{Gr}}\DeclareMathOperator{\Imm}{\mathrm{Im}}\DeclareMathOperator{\gr}{\mathrm{gr}}\DeclareMathOperator{\Irr}{\mathrm{Irr}}\DeclareMathOperator{\Sp}{\mathrm{Sp}}\DeclareMathOperator{\Hom}{\mathrm{Hom}}
\newcounter{AP}
\begin{document}
\author{Ivan Penkov}\address{Ivan Penkov (on leave from Jacobs University Bremen): Yale University, Department of Mathematics 10, Hillhouse Ave., New Haven, CT 06510, USA}\email{i.penkov@jacobs-university.de}\author{Alexey Petukhov}\address{Alexey Petukhov: Institute for information transmission problems, Bolshoy Karetniy 19-1, Moscow 127994, Russia}
\email{alex--2@yandex.ru}
\title{On ideals in the enveloping algebra of a locally simple Lie  algebra}\maketitle
\begin{abstract}We study (two-sided) ideals $I$ in the enveloping algebra $\U(\frak
g_\infty)$ of an infinite-dimensional Lie algebra $\frak g_\infty$
obtained as the union (equivalently, direct limit) of an arbitrary
chain of embeddings of simple finite-dimensional Lie
algebras\begin{center}$\frak g_1\to\frak g_2\to...\to\frak
g_n\to...$\end{center}with $\lim\limits_{n\to\infty}\dim\frak
g_n=\infty$. Our main result is an explicit description of the
zero-sets of
 the corresponding graded ideals $\gr I$. We use this description and
results of A.~Zhilinskii to prove Baranov's conjecture that, if
$\frak g_\infty$ is not diagonal in the sense of A. Baranov and A.
Zhilinskii, then $\U(\frak g_\infty)$ admits a single non-zero
proper ideal: the augmentation ideal.

Our study is based on a complete description of the radical Poisson
ideals in ${\bf S}^\cdot(\frak g_\infty)$ and their zero-sets. We
then discuss in detail integrable ideals of $\U(\frak g_\infty)$,
i.e. ideals $I\subset\U (\frak g_\infty)$ for which $I\cap\U (\frak
g_n)$ is an intersection of ideals of finite-codimension in
$\U(\frak g_n)$ for any $n\ge 1$. We present a classification of
prime integrable ideals based on work of A.~Zhilinskii. For $\frak
g_\infty\cong\frak{sl}_\infty, \frak{so}_\infty$, all zero-sets of
radical Poisson ideals of ${\bf S}^\cdot(\frak g_\infty)$ arise from
prime integrable ideals of $\U(\frak g_\infty)$. For $\frak
g_\infty\cong\frak{sp}_\infty$ only ``half'' of the zero-sets of
Poisson ideals ${\bf S}^\cdot(\frak g_\infty)$ arise from integrable
ideals of $\U(\frak g_\infty)$.


\end{abstract}
\section{Introduction}\label{Sintro}
We work over an algebraically closed field of characteristic zero. A {\it locally finite Lie algebra} is by definition a Lie algebra isomorphic to the limit of some system of finite-dimensional Lie algebras~\cite{BS}. In what follows we restrict ourselves to locally finite Lie algebras $\frak g_\infty$ defined as direct limits of respective sequences of embeddings of simple Lie algebras\begin{center}{}\hspace{119pt}$\frak g_1\to\frak g_2\to...\to\frak g_n\to...$\hspace{119pt}(1)\end{center}such that $\lim\limits_{n\to\infty}\dim\frak g_n=\infty$; for brevity we refer just to these Lie algebras as {\it locally simple} Lie algebras.

Locally simple Lie algebras are natural infinite-dimensional analogs of finite-dimensional simple Lie algebras, and their representation theory is interesting and challenging.

In this paper we study ideals in the universal enveloping algebra $\U(\frak g_\infty)$ of a locally simple Lie algebra $\frak g_\infty$, in particular annihilators of simple $\frak g_\infty$-modules, i.e. primitive ideals. The structure of ideals in  $\U(\frak g_\infty)$ differs significantly from the structure of ideals in the enveloping algebra of a simple finite-dimensional Lie algebra. This is not surprising once one observes that the center of $\U(\frak g_\infty)$ consists of constants only.

The theory of ideals in $\U(\frak g_\infty)$ has been initiated by A.~Zhilinskii in~\cite{Zh2}, \cite{Zh1}. His method is to study ``coherent local systems'' of finite-dimensional modules of the sequence of Lie algebras $\frak g_n$ and to construct ideals in $\U(\frak g_\infty)$ whose intersections with $\U(\frak g_n)$ are the joint annihilators of the modules at the $n$-th level of the respective local systems. We call the ideals of $\U(\frak g_\infty)$ arising in this way {\it integrable} ideals.

A central idea of the present paper is to gain information about
ideals in $\U(\frak g_\infty)$ by studying their associated
``varieties''. At first this leads to the study of radical Poisson
ideals in the symmetric algebra {\bf S}$^\cdot(\frak g_\infty)$. Our first notable result is that {\bf
S}$^\cdot(\frak g_\infty)$ admits a non-zero Poisson ideal
$J$ of locally infinite codimension (i.e. such that $J\cap${\bf
S}$^\cdot(\frak g_n)$ is of infinite codimension in {\bf
S}$^\cdot(\frak g_n)$ for almost all $n$) if and only if $\frak
g_\infty$ is isomorphic to one of the three locally simple Lie
algebras $\frak{sl}_\infty, \frak{so}_\infty, \frak{sp}_\infty$.

This result, together with results of A.~Zhilnskii~\cite{Zh1},
yields the following corollary: if $\frak g$ is not isomorphic to
$\frak{sl}_\infty, \frak{so}_\infty, \frak{sp}_\infty$, and
$\U(\frak g_\infty)$ has a non-zero proper ideal $I$ which does not
coincide with the augmentation ideal, then $I$ is of locally finite
codimension (i.e. $I\cap\U (\frak g_n)$ is of finite codimension in
$\U(\frak g_n)$ for all $n$) and $\frak g_\infty$ is diagonal.
Diagonal locally simple Lie algebras form a natural class of locally
simple Lie algebras which has been introduced by A.~Baranov and A.
Zhilinskii; in fact, these authors have given an intricate and
elegant classification of diagonal locally simple Lie
algebras~\cite{BZh}. Moreover, the above corollary had been
conjectured by A.~Baranov.

As a second notable result on Poisson ideals we compute the set of
zeros $\Var(J)$ (i.e. the associated ``variety'') of any Poisson ideal $J\subset${\bf
S}$^\cdot(\frak g_\infty)$ for $\frak g_\infty=\frak{sl}_\infty,
\frak{so}_\infty, \frak{sp}_\infty$.

We then pass to the study of ideals in $\U(\frak g_\infty)$ for
$\frak g_\infty=\frak{sl}_\infty, \frak{so}_\infty,
\frak{sp}_\infty$. In Section~\ref{Sls} we review Zhilnskii's
results and, most importantly, his classification of irreducible
coherent local systems. Zhilinskii's results yield an explicit
description of all prime integrable ideals in $\U(\frak g_\infty)$,
Theorem~\ref{Tclid}.

We finally describe the associated ``varieties'' of integrable
ideals. In particular, we show that if $\frak
g_\infty=\frak{sl}_\infty, \frak{so}_\infty$ and $J$ is a radical
Poisson ideal of ${\bf S}^\cdot(\frak g_\infty)$, there exists a
prime integrable ideal $I$ of $\U(\frak g_\infty)$ such that
$\Var(I)=\Var(J)$. For $\frak g_\infty=\frak{sp}_\infty$ the
situation is different: given an ideal $J$, there exists an ideal
$I$ of $\U(\frak g_\infty)$ such that $\Var(I)=\Var(J)$, but $I$ is
not necessarily integrable.

We thank A.~Baranov for communicating to us his conjecture and also
the results of A.~Zhilinskii. Ivan Penkov acknowledges support from the DFG via SPP 3188 ``Darstellungstheorie''.

\section{Preliminaries}\label{Spre}
We fix an algebraically closed field $\mathbb F$ of characteristic zero. All
vector spaces (including Lie algebras) are assumed to be defined
over $\mathbb F$. If $V$ is a vector space, $V^*$ stands for the
dual space Hom$_\mathbb F(V, \mathbb F)$. All varieties we consider are algebraic varieties over $\mathbb F$ (with Zariski topology). When considering locally finite Lie algebras or their enveloping algebras we assume
that any given sequence (1) consists of inclusions, so we can freely
interchange $\lim\limits_{\to}\frak g_n$ with $\cup_n\frak g_n$ and $\lim\limits_{\to}\U(\frak g_n)$ with $\cup_n\U(\frak g_n)$.

There is no classification of general locally simple Lie algebras: a
classification is only available for the so called diagonal locally
simple Lie algebras, see~\cite{BZh}. Among diagonal locally simple
Lie algebras a prominent role is played by the three simple Lie
algebras $\frak{sl}_\infty, \frak{so}_\infty$ and $\frak{sp}_\infty$
which can be defined as unions of the respective chains of
inclusions of classical finite-dimensional Lie algebras of types
$\frak{sl}, \frak{so},$ or $\frak{sp}$ under the obvious ``left
upper-corner inclusions''. An important result, see~\cite{B}
or~\cite{BS}, states that, up to isomorphism, these three Lie
algebras are the only locally simple finitary Lie algebras, i.e.
locally simple Lie algebras which admit a countable-dimensional
faithful module with a basis such that the endomorphism arising from
each element of the Lie algebra is given by a matrix with finitely
many non-zero entries.

Let $G$ be a connected algebraic group with Lie algebra, $\frak g$
and $I\subset\U (\frak g)$ be an ideal in the enveloping algebra
$\U(\frak g)$ of $\frak g$ (by an ideal in a ring we mean a proper
two-sided ideal). The degree filtration \{U$(\frak g)^{\le
d}\}_{d\in\mathbb Z_{\ge0}}$ on
 $\U(\frak g)$ defines the filtration $\{I\cap \U (\frak g)^{\le
d}\}_{d\in\mathbb Z_{\ge0}}$ on $I$. The associated graded object
gr$I:=\oplus_d((I\cap\U (\frak g)^{\le d})/(I\cap \U (\frak g)^{\le
d-1}))$ is a $G$-stable ideal of gr~U$(\frak g)=${\bf
S}$^\cdot(\frak g)$. We denote the set of zeros of $\gr I$ in $\frak g^*$ by $\Var(I)\subset\frak g^*$. The variety $\Var(I)$ is a $G$-stable subvariety of $\frak g^*$ and is, by definition,
 the {\it associated variety} of $I$. The ideal $I$ is
of finite codimension in $\U(\frak g)$ if and only if gr$I$ is of
finite codimension in {\bf S}$^\cdot(\frak g$). Moreover, gr$I$ is
of finite codimension in {\bf S}$^\cdot(\frak g$) if and only if
$\Var(I)=0$. We note that the radical rad(gr$I$) of gr$I$ is also $G$-stable.

It is well known that {\bf S}$^\cdot(\frak g)$ is a Poisson algebra, i.e. {\bf S}$^\cdot(\frak g)$ has an $\mathbb F$-bilinear operation
\begin{center}\{$~\!\cdot, \cdot$\}: {\bf S}$^\cdot(\frak g)\times${\bf S}$^\cdot(\frak
g)\to${\bf S}$^\cdot(\frak g)$\end{center}which is a Lie bracket and
is compatible with multiplication. An ideal $J\subset${\bf
S}$^\cdot(\frak g)$ is {\it Poisson} whenever $\{f, J\}\subset J$
for any $f\in${\bf S}$^\cdot(\frak g)$. Assume that $\frak g$ is
finite-dimensional and semisimple and let $G$ be the adjoint group
of $\frak g$. Then the radical Poisson ideals of {\bf S}$^\cdot(\frak g)$
are in one-to-one correspondence with the $G$-stable Zariski-closed
subsets of $\frak g^*$. A description of such sets is presented
in~\cite{Bor}.

Let $\frak g_\infty$ be a locally simple Lie algebra. The same
arguments as in the previous paragraph assign to any ideal
$I\subset\U (\frak g_\infty)$ a radical ideal rad(gr$I)\subset${\bf
S}$^\cdot(\frak g_\infty)$, which is stable under the adjoint action
of $\frak g_\infty$ on {\bf S}$^\cdot(\frak g_\infty)$; in what
follows we call $\frak g_\infty$-stable ideals of {\bf
S}$^\cdot(\frak g_\infty)$ Poisson. We denote the set of zeros of
gr$I$ in $\frak g_\infty^*$ by $\Var(I)$ and refer to $\Var(I)$ as
the associated ``variety'' of $I$.

Note that $\Var(I)$ is a proj-variety, i.e. an inverse limit of algebraic varieties. Indeed, fix a sequence (1) with $\frak g_\infty=\lim\limits_{\to}\frak g_n$ and let
$\overline{\pr_{\frak g_n}\Var(I)}\subset\frak g_n^*$
be the closure of the image of $\Var(I$) under the natural projection
pr$_{\frak g_n}: \frak g_\infty^*\to\frak g_n^*$; by definition
$\overline{\pr_{\frak g_n}\Var(I)}\subset\frak g_n^*$
is the set of zeros of $(\gr I)\cap${\bf S}$^\cdot(\frak g_n)$ in $\frak g_n^*$. The space $\frak g_\infty^*$ equals the inverse limit
$\varprojlim \frak g_n^*$, and therefore $\Var(I)\subset\frak
g_\infty^*$ is the inverse limit of the algebraic varieties
$\overline{\pr_{\frak g_n}\Var(I)}$.

An ideal $I$ is of {\it locally finite codimension} in $\U(\frak g_\infty)$ (i.e.
$I\cap\U (\frak g_n)$ is of finite codimension in $\U(\frak g_n)$ for
all $n$) if and only if the ideal gr$I$ is of {\it locally finite
codimension} in {\bf S}$^\cdot(\frak g_\infty$) (i.e.
(gr$I)\cap${\bf S}$^\cdot(\frak g_n)$ is of finite codimension in
{\bf S}$^\cdot(\frak g_n)$ for all $n$). Furthermore, gr$I$ is of
locally finite codimension in {\bf S}$^\cdot(\frak g_\infty$) if and
only if $\Var(I)=0$. We call an ideal $I$ of {\it locally infinite codimension} if $I$ is
not of locally finite codimension.

The above discussion implies in particular the following.
\begin{proposition}\label{Clfcd}Assume that $\frak g_\infty$ is a locally simple Lie algebra. If $\U(\frak g_\infty)$ admits a non-zero ideal of locally infinite codimension, then {\bf S}$^\cdot(\frak g_\infty)$ admits a non-zero Poisson ideal of locally infinite codimension.\end{proposition}
\section{Poisson ideals: statement of results}\label{Spintro}
Our first main result is the following theorem.
\begin{theorem}\label{Tlinf}If {\bf S}$^\cdot(\frak g_\infty)$ admits a non-zero Poisson ideal of locally infinite codimension, then $\frak g_\infty\cong\frak{sl}_\infty, \frak{so}_\infty$, $\frak{sp}_\infty$.\end{theorem}
\begin{corollary}\label{Ctlinf}If $\U(\frak g_\infty)$ admits a non-zero ideal of locally infinite codimension, then $\frak g_\infty\cong\frak{sl}_\infty, \frak{so}_\infty$, $\frak{sp}_\infty$.\end{corollary}
\begin{proof}The algebra {\bf S}$^\cdot(\frak g_\infty)$ admits a non-zero Poisson ideal $J$ of locally infinite codimension by Proposition~\ref{Clfcd}, hence $\frak g_\infty\cong\frak{sl}_\infty,$ $\frak{so}_\infty$, $\frak{sp}_\infty$ by Theorem~\ref{Tlinf}.\end{proof}
Fix now a Lie algebra $\frak g_\infty=\frak{sl}_\infty, \frak{so}_\infty$, $\frak{sp}_\infty$ together with a chain (1) such that $\lim\limits_{\to}\frak g_n=\frak g_\infty$. Without loss of generality we assume that for $n\ge 3$ all $\frak g_n$ are simple and of the same type A, B, C, or D, and that $\rk\frak g_n=n$. By $V_n$ we denote a natural representation of $\frak g_n$ (for $\frak g_n$ of type A there are two choices of $V_n$ up to isomorphism). We further assume that, for $n\ge 3$, $V_{n+1}$ considered as a $\frak g_n$-module is isomorphic to $V_n$ plus a trivial module.

Set
\begin{center}\hspace{27pt}$\frak g_n^{\le r}:=\{x\in\frak g_n\mid$ there exists
$\lambda\in\mathbb F$ such that $\rk (X-\lambda$Id$_{V_n})\le r\}$,
\hspace{27pt}(2)\end{center}where $X$ is considered as a linear
operator on $V_n$. Note that $\frak g_n^{\le r}$ is a Zariski closed
subset of $\frak g_n$. Choosing compatible identifications $\frak
g_n\cong\frak g_n^*$, we can assume that $\frak g_n^{\le
r}\subset\frak g_n^*$. Furthermore, for $\frak
g_\infty\cong\frak{sl}_\infty, \frak{so}_\infty, \frak{sp}_\infty$
one can check directly that the projection $\frak g_{n+1}^*\to\frak
g_n^*$ maps $\frak g_{n+1}^{\le r}$ surjectively onto $\frak
g_n^{\le r}$. This yields a well-defined limit of algebraic
varieties $\varprojlim \frak g_n^{\le r}$ which we denote by $\frak
s_\infty^{\le r}$, where $\frak s$ is an abbreviation for
$\frak{sl}, \frak{so}$ or $\frak{sp}$.

The radical ideals $J_n^{\le r}$ of ${\bf S}^\cdot(\frak g_n)$ with
respective zero-sets $\frak g_n^{\le r}\subset\frak g_n^*$ form a
chain whose union we denote by $J^{\le r}$. The ideal $J^{\le r}$ is
a radical Poisson ideal of {\bf S}$^\cdot(\frak g_\infty)$. It turns
out that any non-zero radical Poisson ideal of {\bf S}$^\cdot(\frak
g_\infty)$ is of this form.

\begin{theorem}\label{Tlinfrk}Let $\frak g_\infty=\frak{sl}_\infty, \frak{so}_\infty, \frak{sp}_\infty$ and $J\subset{\bf\mathrm S}^\cdot(\frak g_\infty)$ be a non-zero radical Poisson ideal. Then $J=J^{\le r}$ for some $r\in\mathbb Z_{\ge0}$.\end{theorem}
Theorems~\ref{Tlinf} and~\ref{Tlinfrk} are proved in
Section~\ref{Ssch}.
\section{Auxiliary results}\label{Sppsl}
\label{Sco} Throughout the paper $V$ denotes a finite-dimensional vector space of dimension $d_V$. In order to consider all simple classical groups simultaneously, we
use S (respectively, $\frak s$) as an abbreviation for SL, SO, Sp (respectively, $\frak{sl}, \frak{so}$,
$\frak{sp}$) and consider three different cases. In the case S=SL we fix the zero bilinear form on the space $V$ and set S$(V):=$SL$(V)$,
$\frak s(V):=\frak{sl}(V)$. In the case S=SO we fix a nondegenerate symmetric bilinear form on $V$ and set S$(V):=$SO$(V)$, $\frak s(V):=\frak{so}(V)$. In the case
S=Sp we assume that $d_V$ is even and fix a
nondegenerate antisymmetric bilinear form on $V$. Then S$(V):=$Sp$(V)$, $\frak
s(V):=\frak{sp}(V)$.

In this section $G$ denotes a connected simple subgroup of S$(V)$ with Lie algebra $\frak g\subset\frak s(V)$. We start with some general statements about $G$-orbits in $\frak g^*$ (i.e. about coadjoint orbits). We identify $\frak g$ and $\frak g^*$ via the Cartan-Killing form. Fix $e\subset\frak g^*$. Denote the $G$-orbit of $e$ in $\frak g^*$ by $\EuScript O[e]$. Let $\underline{\EuScript O}[e]$ be the unique closed $G$-orbit  of the closure $\overline{\EuScript O[e]}$ in $\frak g^*$. By the Luna Slice Theorem~\cite{VP}, there exists a $G$-equivariant morphism $\EuScript O[e]\to\underline{\EuScript O}[e]$. Assume that $\underline{\EuScript O}[e]\ne 0$, i.e. $e$ is not nilpotent. Let $h\in\underline{\EuScript O}[e]$ be a semisimple element. Then the centralizer of $h$ in $\frak g$ is a Levi subalgebra~\footnote{Under a Levi subalgebra of a parabolic subalgebra $\frak p\subset\frak g$ we understand a maximal reductive in $\frak g$ subalgebra of $\frak p$.} of some parabolic subalgebra $\frak p\subset\frak g$. Let $P\subset G$ be a parabolic subgroup with Lie algebra $\frak p$. Then there exists a finite $G$-equivariant covering $\tilde{\EuScript O}\to\EuScript O[e]$ and a $G$-equivariant morphism $\tilde{\EuScript O}\to G/P$.

Assume now that $\underline{\EuScript O}[e]=\{0\}$, i.e. that $e$ is nilpotent. By the Jacobson-Morozov Theorem there exist elements $h, f$ such that $$[h, e]=2e, [h, f]=-2f, [e, f]=h,$$ i.e. such that $\{ e, h, f\}$ is an $\frak{sl}_2$-triple. The element $h$ is rational semisimple. Hence $\frak g$ splits into the direct sum of ad$h$-eigenspaces $\oplus_{i\in\mathbb Q}\frak g_i$ with rational eigenvalues. The direct sum $\frak g_h^+:=\oplus_{i\ge 0}\frak g_i\subset\frak g$ is a parabolic subalgebra of $\frak g$ and we denote it by $\frak p_e$. The subalgebra $\frak p_e$ is determined by $e$. By $P_e\subset G$ we denote the parabolic subgroup with the Lie algebra $\frak p_e$. There is a $G$-equivariant morphism $\EuScript O[e]\to G/P_e$.

The above discussion is summarized in the following lemma.
\begin{lemma}\label{Lcgr}A suitable finite covering of any coadjoint orbit
admits a $G$-equivariant (and thus surjective) morphism to $G/P$,
where $P$ is a maximal parabolic subgroup.\end{lemma} For $\frak
g=\frak{sl}(V)$, a quotient $G/P$ is nothing but a Grassmannian
Gr$(r; V)$ for some $r<d_V$. For $\frak g=\frak{so}(V)$ and
$\frak{sp}(V)$ a quotient $G/P$ is an irreducible component of the
variety Gr$(^0r; V)$ of isotropic subspaces in $V$ of dimension $r$
for some $r\le \frac{d_V}2$. The variety Gr$(^0r; V$) is irreducible
unless $\frak g=\frak{so}(V)$ and $r=\frac{d_V}2$. In this latter
case Gr$(^0r; V)$ has two irreducible components which are
isomorphic as varieties. More generally, for $k<r<d_V$, we denote by
Gr$(^kr; V)$ the variety of subspaces of $V$ of dimension $r$ on
which the restriction of the fixed form on $V$ has rank $k$.


For $X\in$End$V$ we denote by $V^X_\lambda\subset V$ the generalized
eigenspace of $X$ with eigenvalue $\lambda\in\mathbb F$.
Furthermore, $\frak g\cdot V$ stands for the sum of the non-trivial
simple $\frak g$-submodules of $V$.
\begin{proposition}\label{Csl}Assume that there is an S$(V)$-orbit
$\EuScript O$ in $\frak s(V)^*$ such that its image in $\frak g^*$
under the canonical projection $\frak s(V)^*\to\frak g^*$ is not
dense. Then\begin{center}dim$(\frak g\cdot
V)<$2(dim$G$--$\rk G$)($\rk G$+1) or 2dim$G+2\ge
d_V$.\end{center}\end{proposition} In order to prove
Proposition~\ref{Csl} we need several preliminary statements.
\begin{lemma}\label{Lla1}\label{Llaso}Let $V'$ be any subspace in $V$ of dimension $d_{V'}$.\\a) Assume that $\frak s=\frak{sl}$, $\frak{sp}$. If $\dim(V^r\cap V')\ge 1$ for any isotropic subspace $V^r\subset V$
of dimension $r$, then $r+d_{V'}>d_V$.\\b) Assume that $\frak
s=\frak{so}$ and $r<\frac{d_V}{2}$. If $\dim(V^r\cap V')\ge 1$ for
any isotropic subspace $V^r\subset V$ of dimension $r$, then
$r+d_{V'}>d_V$.\\c) Assume that $\frak s=\frak{so}$, $d_V$ is even,
and $r=\frac{d_V}{2}$. If $\dim(V^r\cap V')\ge 1$ for any isotropic
subspace $V^r\subset V$ of dimension $r$ from some irreducible
component of Gr$(^0r; V)$, then either $r+d_{V'}>d_V$, or
$r+d_{V'}$=$d_V$ and $V'$ is isotropic.\end{lemma}
\begin{proof}Exercise in linear algebra.\end{proof}
Let $Z$ be a $G$-variety. We denote the maximal dimension of a
$G$-orbit on $Z$ by m$_G(Z)$. For a subvariety $Y\subset Z$ we set
\begin{center}$GY:=\{z\in Z\mid z=gy$ for some $y\in Y$ and $g\in
G\}$.\end{center}
\begin{lemma}\label{Lpo}a) Assume that m$_G(\mathbb P(V))<\dim G$. Then \begin{center}$\dim(\frak g\cdot V)<2(\dim G$--rank$G$)(rank$G$+1).\end{center}b)\label{Lisoco} Let $\frak s=\frak{so}$. Assume that
\begin{center}m$_G(\mathbb P(V))=$dim$G$ and m$_G($Gr$(^01; V))<$dim$G$.\end{center}Then\begin{center}2dim$G+2\ge d_V$.\end{center}\end{lemma}
\begin{proof}The statement of part a) is the main result of~\cite{AnPo}. In part
b) Gr$(^01; V)=\mathbb P(\mathcal V)$, where $\mathcal V$ is the set
of isotropic vectors in $V$. The inequality m$_G(\mathbb P(\mathcal
V))<$dim$G$ implies that the stabilizer of a generic point of $\mathbb P(\mathcal V)$ under
the action of $G$ has positive dimension. Therefore there exists $A\in\frak g$ such that $\EuScript O[A]$ intersects $\frak g_x$ for all $x$ in an open subset of $\mathbb P(\mathcal V)$. Then, for some eigenvalue
$\lambda$ of $A$, we have $\mathcal V\subset\overline{GV^A_\lambda}$, where
\begin{center}$V^A_\lambda:=$Ker$(A-\lambda$Id$_V)^{d_V}$\end{center} and $\overline{GV^A_\lambda}$ is the closure of $GV^A_\lambda$ in $V$. Furthermore, the equality
m$_G(\mathbb P(V))=$dim$G$ implies that $\mathcal V$ coincides  with
$\overline{GV^A_\lambda}$. Therefore $V^A_\lambda\subset\mathcal V$,
and the space $V^A_\lambda$ is isotropic. In particular,
dim$V^A_\lambda\le\frac{d_V}{2}$, and consequently $\dim
G\ge\frac{d_V}{2}$-1.\end{proof}
\begin{lemma}\label{LGr}\label{Lisoso} Let
\begin{center}$2\le r\le\frac{d_V}{2}$, m$_G(\mathbb P(V))=$dim$G$ and m$_G($Gr$(^0r;
V))<$dim$G$\end{center}for some $r$. Then \begin{center}2dim$G+2\ge
d_V$.\end{center}\end{lemma}
\begin{proof}Let $x\in$Gr$(^0r; V)$ denote a point and $V(x)$ be the corresponding $r$-dimensional space. Let $G_x$ be the stabilizer of $x$ in $G$; $G_x$ acts on $V(x)$ and, if $x$ is generic, we have \begin{center}m$_{G_x}(\mathbb P(V(x)))=$dim$G_x$.\end{center}

Let $p\in V(x)$ be a non-zero vector, $\langle p\rangle\subset V(x)$ be the line generated by $p$, and $\frak gp\subset V$ be the tangent space to $Gp$ in $p$. Then dim$((V(x)/\langle p\rangle)\cap\frak g\langle
p\rangle)\ge$dim$G_x$, where $\frak g\langle p\rangle$ is the image
of $\frak gp$ in $V/\langle p\rangle$. Hence,
for a generic $\langle p\rangle\in$Gr$(^01; V)$ and any
\begin{center}$\tilde V\in$Gr$(^0(r-1); \langle
p\rangle^\bot/\langle p\rangle)$\end{center} we have dim$(\tilde
V\cap\frak g\langle p\rangle)\ge$dim$G_x$. In particular,
dim$(\tilde V\cap\frak g\langle p\rangle)\ge1.$ Therefore, by
Lemma~\ref{Lla1} (applied to $V'=\frak g\langle p\rangle$ and
$V=\langle p\rangle^\bot/\langle p\rangle$) we obtain
\begin{center}$r-1$+dim$\frak g\ge$dim$(\langle p\rangle^\bot/\langle p\rangle)$.\end{center}
As dim$(\langle p\rangle^\bot/\langle p\rangle)\ge d_V-2$ and
$r\le\frac{d_V}2$, we conclude that $2\dim G+2\ge d_V$.\end{proof}
\begin{proposition}[{~\cite[text after Prop.~3]{Vi}}]\label{Pvi} Let $Z$ be an irreducible $G$-symplectic variety with a moment map $\Phi: Z\to\frak g^*$. Then the dimension of $\Phi(Z)$ equals~m$_G(Z)$.\end{proposition}
\begin{proof}[Proof of Proposition~\ref{Csl}]Since we can consider the morphism $\EuScript O\to\frak g^*$ as a moment map, Proposition~\ref{Pvi} implies that the image of $\EuScript O$ in $\frak g^*$ is not dense if and only if \begin{center}m$_G(\EuScript O)<$dim$G$.\end{center}Furthermore, according to Lemma~\ref{Lcgr}, for some $r$, $1\le r\le\frac{d_V}2$, there exists an S$(V)$-equivariant morphism from a finite covering of $\EuScript O$ onto Gr$(^0r; V)$. Therefore \begin{center}m$_G($Gr$(^0r; V))<$dim$G$.\end{center}

If m$_G($Gr$(^01; V))<$dim$G$, the claim of Proposition~\ref{Csl}
follows from Lemma~\ref{Lpo}. Indeed, if also m$_G(\mathbb
P(V))<\dim G$, we apply Lemma~\ref{Lpo} a), and if m$_G(\mathbb
P(V))=\dim G$, then $\frak s=\frak{so}$ (as $\mathbb P(V)\ne
\Gr(^01; V))$ and we apply Lemma~\ref{Lpo} b).

If \begin{center}m$_G($Gr$(^01; V))=$dim$G$,\end{center} then
necessarily m$_G(\mathbb P(V))=$dim$G$ and $r\ge2$. In this case the
claim of Proposition~\ref{Csl} follows from
Lemma~\ref{LGr}.\end{proof}

Proposition~\ref{Csl} is used in a crucial way in the proof of Theorem~\ref{Tlinf}. The results in the remaining part of this section are necessary for the proof of Theorem~\ref{Tlinfrk}.

In what follows $W$ denotes a subspace of $V$ of dimension $d_W$. If
$\frak s=\frak{so}, \frak{sp}$, we assume that the restriction of
the fixed form on $V$ is nondegenerate on $W$. This yields
embeddings S$(W)\to$S$(V)$ and $\frak s(W)\to\frak s(V)$. For $\frak
s=\frak{sl}$ an embedding S$(W)\to$S$(V$) is determined by a choice
of complement to $W$ in $V$. Moreover, in all three cases we have an
orthogonal (with respect to the Cartan-Killing form on $\frak s(V)$)
projection $\phi:\frak s(V)\to\frak s(W)$. It is easy to see that
$\phi(X)=\pr_W\circ(X|_W)$, where $X\in\frak s(V)$ is viewed as an
element of $\End V$ and $\pr_W: V\to W$ is the orthogonal projection
for $\frak s=\frak{so}, \frak{sp}$, and respectively the projection
along the fixed complement of $W$ in $V$ for $\frak s=\frak{sl}$. We
fix the embedding $\frak s(W)\to\frak s(V)$, and $\phi: \frak
s(V)\to\frak s(W)$ stands for the corresponding projection.

Recall that $\EuScript O[X]$ denotes the S$(V)$-orbit in $\frak
s(V)$ of $X\in\frak s(V)$.
\begin{lemma}\label{Lrkmaxsp}Let
$X\in\frak s(V)$. Then\\a)
 $\rk \phi(X)\le\rk X$;\\b) if $d_W>\rk X$
, the image $\phi(\EuScript O[X])$ contains an element of rank $\rk
X$.\end{lemma}
\begin{proof} Part a) follows immediately from the formula
\begin{center}$\phi(X)=$pr$_W\circ (X|_W)$.\end{center}

To prove part b) we consider a generic subspace $\tilde W\subset W$
of dimension $d_W$. Note that for $\frak s=\frak{so}, \frak{sp}$ the
restriction of the form to $\tilde W$ is nondegenerate. Furthermore,
$\rk X|_{\tilde W}=\rk X$, and $\tilde W^\bot\cap\Imm X$=0 where
$\Imm X$ is the image of $X\in\End V$. Therefore $\rk(\pr_{\tilde
W}\circ X|_{\tilde W})=\rk X$.

Since there exists $g\in$S$(V)$ with $g(W)=\tilde W$, we have
\begin{center}$\rk \phi(g(X))=\rk X$.\end{center}
\end{proof}
\begin{proof}The proof is very similar to the proof of Lemma~\ref{Lrkmaxsp}.\end{proof}
\begin{lemma}\label{Lbs}Assume $r\le\frac{d_W}2$. Then the set $\frak s(V)^{\le r}$ defined by (2) is the largest S$(V)$-invariant subset of the preimage $\phi^{-1}(\frak s(W)^{\le r})$.\end{lemma}
\begin{proof} Exercise in linear algebra.\end{proof}
\begin{lemma}\label{Lglsp}Consider the projection $\varphi: \frak{sp}(V)^*\to\frak{gl}(V^{iso})^*$ dual to the embedding $\frak{gl}(V^{iso})\to\frak{sp}(V)$, where $V^{iso}$ is a maximal isotropic subspace of
$V$.
Assume $r\le\frac{d_V}2$. Then the set $\frak{sp}(V)^{\le r}$
defined by (2) is the largest $\Sp(V)$-invariant subset of the
preimage $\varphi^{-1}(\frak{sl}(V^{iso})^{\le r})$.\end{lemma}
\begin{proof} Exercise in linear algebra.\end{proof}
\begin{lemma}\label{Lla2}Let $x\in$Gr$(r; V)$ and let $V(x)\subset V$ be the corresponding subspace. Let $W'$ be the S$(W)$-stable complement to $W$ in $V$. Assume that $d_V\ge2d_W$, $d_W\le r\le d_V-d_W$, and dim$(V(x)\cap W')=r-d_W$. Then the stabilizer of $x\in$Gr$(r; V)$ in SL$(W)$ (and therefore also in SO$(W)$, Sp$(W$)) is trivial.\end{lemma}
\begin{proof} Exercise in linear algebra.\end{proof}
For a subset $S\subset\mathbb F$ we set $V^X_S:=\oplus_{\lambda\in
S}V^X_\lambda$. We set also $\hat 0:=\mathbb F\backslash\{0\}$; then
$V=V_0^X\oplus V_{\hat 0}^X$. Let $X_{nn}:=X|_{V_0},
X_r:=X|_{V_{\hat 0}}$, and $X_r=X_s+X_{ns}$ be the Jordan
decomposition as sum of commuting semisimple and nilpotent elements.
The decomposition $V=V_0^X\oplus V_{\hat 0}^X$ allows as to consider
all four operators $X_{nn}, X_r, X_s, X_{ns}$ as endomorphisms of
$V$. Furthermore,
\begin{center}$\rk(X_s+X_{ns})=\rk X_s,\hspace{10pt}\rk(X_s+X_{nn})=\rk X$,\end{center}and $X_s+X_{nn}\in\overline{\EuScript O[X]}$.
\begin{definition}We say that $X\in$End$V$ is {\it rank-reduced} whenever \begin{center}$X_{nn}^2=0$ and $X_{ns}=0$.\end{center}\end{definition}

\begin{lemma}\label{Lrr}If for some $X\in\frak s(V)$ we have 2$\rk X\le d_V$, then there
exists a (unique up to conjugation) rank-reduced element $X'\in\frak
s(V)$ such that $X'\subset\overline{\EuScript O[X]}$ and $\rk X=\rk
X'$.\end{lemma}
\begin{proof}The condition $(X'_{nn})^2=0$ means that  $X_{nn}'$ is
nilpotent and the sizes of its Jordan blocks are at most 2$\times2$.
Moreover, the number of non-zero Jordan blocks of $X_{nn}'$ equals
$\rk X_{nn}'$. In particular, all nilpotent rank-reduced
endomorphisms in $\frak s(V)$ of fixed rank are conjugate (possibly,
by outer automorphisms of $\frak s(V)$). Since $2\rk X_{nn}\le
d_V$, there exists a nilpotent rank-reduced endomorphism $X_{rr}$ of
rank equal to $\rk X$ and such that $X_{rr}\in\overline{\EuScript
O[X_{nn}]}$~\cite{CM}. A rank-reduced endomorphism $X'\in\frak
s(V)$ such that $X_s'=X_s$ and $X'_{nn}$ is conjugate to $X_{rr}$,
is as desired. By construction, $X'$ is unique up to conjugation.
\end{proof}
\begin{lemma}\label{Lsl1}Assume $d_V>3d_W$. If for some $X\in\frak s(V)$ the image $\phi(\EuScript O[X])$ is not dense in $\frak s(W)$, then there exists a unique $\lambda\in\mathbb F$ such that \begin{center}$\rk(X-\lambda$Id$_V)<d_W$.\end{center}\end{lemma}
\begin{proof}
Set $d_V(\mu):=$dim$V^X_\mu$ and $d_V(S):=\sum_{\mu\in S}d_V(\mu)$ for any subset
$S\subset\mathbb F$. We have an S$(V)$-equivariant morphism
\begin{center}$\EuScript O[X]\to$Gr$(d_V(S); V),\hspace{10pt}X\longmapsto\oplus_{\mu\in
S}V^X_\mu$.\end{center} As $\phi(\EuScript O[X])$ is not dense in
$\frak s(W)$, the inequality m$_{\mathrm{S}(W)}(\EuScript
O[X])<$dim~S$(W)$ holds. Therefore, for any $S\subset\mathbb F$,
there exists $k$ such that
\begin{center}{}\hspace{90pt}m$_{\mathrm{S}(W)}(\Gr(^kd_V(S); V))<$dim~S$(W)$.\hspace{90pt} (3)\end{center}

We now show that, for any $S\subset\mathbb F$, either $d_V(S)<d_W$
or $d_V(S)>d_V-d_W$. Assume to the contrary that $d_W\le d_V(S)\le
d_V-d_W$. Let $x\in$Gr$(^kd_V(S); V)$ be a generic point and
$V(x)\subset V$ be the corresponding subspace. Then
\begin{center}$\dim(V(x)\cap W')=\dim V(x)-d_W$,\end{center}where
$W'$ is the S$(W)$-stable complement to $W$ in $V$, and
Lemma~\ref{Lla2} implies that the stabilizer of $x$ in S$(W)$ is
trivial. This contradicts (3).

We claim next that there exists $\lambda\in\mathbb F$ such that
$d_V(\lambda)>d_V-d_W$. Indeed, assuming that $d_V(\lambda)<d_W$ for
all $\lambda\in\mathbb F$, we see that the inequality $3d_W<d_V$,
together with the alternative $d_V(S)<d_W$ or $d_V(S)>d_V-d_W$,
implies $d_V(S)<d_W$ for any finite set $S$. This is again a
contradiction.

Fix now $\lambda$ with $d_V(\lambda)>d_V-d_W$. It is easy to see
that $\rk (X-\lambda\Id_V)=d_V(\mathbb C\backslash\lambda)+\rk
(X-\lambda\Id_V)|_{V^X_\lambda}$. Furthermore, we have a
decomposition
$$X-\lambda\mathrm{Id}_V=(X-\lambda\mathrm{Id}_V)_s+(X-\lambda\mathrm{Id}_V)_{ns}+(X-\lambda\mathrm{Id}_V)_{nn}.$$
Let $r'\in\mathbb Z_{\ge0}$ for $\frak s=\frak{sl}$ and
$r'\in2\mathbb Z_{\ge0}$ for $\frak s=\frak{so}, \frak{sp}$. If
$r'<\rk ((X-\lambda$Id$_V)_{nn}|_{V^X_\lambda})$, there exists
$X_{r'}^\lambda\in\overline{\EuScript
O[(X-\lambda\mathrm{Id}_V)_{nn}|_{V^X_\lambda}]}\subset\frak
s(V^X_\lambda)$ with $\rk X_{r'}^\lambda=r'$. Hence, if
 $\rk ((X-\lambda$Id$_V)_s)\le r'\le\rk (X-\lambda$Id$_V$), there
exists
\begin{center}$X_{r'}\in\overline{\EuScript O[X]}\subset\frak s(V)$\end{center}
with $X_{r'}-\lambda$Id=$(X-\lambda$Id$)_s$+$X^{\lambda}_{r'-\mathrm{rk}((X-\lambda\mathrm{Id})_s)}$, and thus $\rk (X_{r'}-\lambda$Id$_V)=r'$.

As, for any $d$ with $d_W\le d\le d_V-d_W$, the equality m$_{\mathrm{S}(W)}($Gr$(^kd; V))=$dim~S$(W)$ holds for some $k$, for any
$r'$ we have either dim~Im$(X_{r'}-\lambda$Id$_V)<d_W$ or
dim~Im$(X_{r'}-\lambda$Id$_V)>d_V-d_W$. Therefore, either
 $\rk (X-\lambda$Id$_V)<d_W$ or $\rk (X-\lambda$Id$_V)_s>d_V-d_W>d_W$. On
the other hand, $\rk (X-\lambda$Id$_V)_s<d_W$. This
implies\begin{center}$\rk(X-\lambda$Id$_V)~<~d_W$.\end{center}\end{proof}
\begin{lemma}\label{Lsl2}Assume $d_V\ge2d_W$, $d_W>r$, and let $X\in\frak s(V)^{\le r}$. If $\rk (X-\lambda$Id$_V)=r$ for some $\lambda\in\mathbb F$, then $\phi(\EuScript O[X])$ is dense
in $\frak s(W)^{\le r}$.\end{lemma} We consider separately the cases
$\frak s=\frak{sl}$ and $\frak s=\frak{so}, \frak{sp}$.
\begin{proof}[Proof of Lemma~\ref{Lsl2} for $\frak s=\frak{sl}$] By Lemma~\ref{Lrr} we can assume that $X-\lambda$Id$_V$ is rank-reduced.
Therefore the linear operator $X-\lambda$Id$_V: V\to V$ is conjugate
to a direct sum
$$\bigoplus\limits_{1\le i\le\mathrm{rk}X_s}A(t_i)\bigoplus\limits_{1\le j\le\mathrm{rk}X_{nn}}B\bigoplus C,$$
where $C$ is the zero operator, and $A(t), B$ are operators of
the following forms:
$$A(t):=\left(\begin{array}{cc}t&0\\0&0\\\end{array}\right),\hspace{10pt}B:=\left(\begin{array}{cc}0&1\\0&0\\\end{array}\right).$$ We denote the elements of the respective bases of the 2-dimensional subspaces of $V$ corresponding to $A(t_i)$ and $B$ by $a_1[i], a_2[i]$ and $b_1[j],
b_2[j]$, and the elements of the basis of the subspace corresponding
to $C$ by $c[k]$. We can further assume that
\begin{center}$W:=$span$\{a_1[i], b_1[j], c[k]\mid
1\le i\le\mathrm{rk}X_s, 1\le j\le\mathrm{rk}X_{nn}, 1\le k\le
d_W-r$\}\end{center} and that the SL$(W)$-invariant complement $W'$ of $W$ is the span of the remaining
elements of the basis $\{a_1[i], a_2[i], b_1[j], b_2[j], c_k\}_{i,
j, k}$ of $V$. Then
$$\phi(\bigoplus_i\EuScript O[A(t_i)]\bigoplus_j\EuScript
O[B]\bigoplus C)\subset\frak{sl}(\bigoplus\limits_{1\le
i\le\mathrm{rk}X_s}\frak{gl}\langle a_1[i],
a_2[i]\rangle\bigoplus\limits_{1\le
j\le\mathrm{rk}X_{nn}}\frak{sl}\langle b_1[i], b_2[i]\rangle).$$
Moreover, $\phi(\bigoplus_i\EuScript O[A(t_i)]\bigoplus_j\EuScript
O[B]\bigoplus C)$ consists of all diagonal matrices of the form
\begin{center}diag$(\lambda_1,..., \lambda_r,0,...,
0),\hspace{10pt}\lambda_1+...+\lambda_r=0.$\end{center} This follows
from the obvious statement that the matrices in $\EuScript O[A(t)]$
for $t\ne 0,$ and in $\EuScript O[B]$ can have arbitrary values in
their left top corner. Therefore $\phi(\EuScript O[X])$ contains all
semisimple elements of a given rank $r=\rk X$. As such elements are
dense in $\frak{sl}(W)^{\le r}$, $\phi(\EuScript O[X])$ is dense in
$\frak{sl}(W)^{\le r}$.\end{proof} In the rest of this section
$\frak s=\frak{so}$, $\frak{sp}$. To prove Lemma~\ref{Lsl2} in this
case, we need some preliminary lemmas.
\begin{lemma}\label{Leig0}Let $X\in\frak s(V)$. Assume that 2$\rk(X-\lambda\Id_V)<d_V$ for some $\lambda\in\mathbb F$. Then $\lambda=0$.\end{lemma}
\begin{proof}Obviously, $\rk (X-\lambda$Id$_V)+d_V(\lambda)\ge d_V$~(recall that $d_V(\lambda)=\dim V_\lambda^X$). Therefore 2$\rk(X-\lambda\Id_V)<d_V$ implies that $2d_V(\lambda)>d_V$.

However, for $\frak{so}, \frak{sp}$, we have $d_V(\lambda)=d_V(-\lambda)$. Therefore $\lambda\ne0$ would mean that $\lambda\ne-\lambda$ and $\dim V_\lambda^X+\dim V_{-\lambda}^X>d_V$, which is false. Hence $\lambda=0$.\end{proof}

Fix $r\in\mathbb Z_{\ge0}$. If $r$ is even we denote by $\frak
s(V)^r_{reg}$ the set of semisimple elements of $\frak s(V)$ which have
rank $r$. If $r$ is odd we denote by $\frak{sp}(V)^r_{reg}$ the set
of elements $X$ of $\frak{sp}(V)$ such that \begin{center}$\rk
X_{s}=r-1$, $X_{ns}=0$, $\rk X_{nn}=1$.\end{center}
In particular, $X_{nn}$ is nilpotent and has a
single non-zero Jordan block of size 2$\times$2.
\begin{lemma}\label{Lrkpre}a) Assume $d_V$ is even. Then $\overline{\frak{sp}(V)^{d_V-1}_{reg}}=\frak{sp}(V)^{\le(d_V-1)}$.\\b) Assume $d_V$ is odd. Then $\overline{\frak{so}(V)^{d_V-1}_{reg}}=\frak{so}(V)^{\le(d_V-1)}$.\end{lemma}
\begin{proof}Exercise in linear algebra.\end{proof}

In connection with the following two lemmas, note that $\rk X$ is
even for any $X\in\frak{so}(V)$. Recall also that on $\frak s(V)$
there is the following partial order: $X_1$ is {\it lower} than
$X_2$ if $\EuScript O[X_1]\subset\overline{\EuScript O[X_2]}$. This
order can be described explicitly in terms of Young diagrams (or
equivalently, in terms of the Jordan normal forms of $X_1, X_2$),
see~\cite{CM}. 

\begin{lemma}\label{Lrkspo} Let $r$ be an odd integer such that $1\le r\le d_V$; assume that $d_V$ is even. Then $\frak{sp}(V)^{\le r}=\overline{\frak{sp}(V)^r_{reg}}$.\end{lemma}
\begin{proof} By definition $\frak{sp}(V)^r_{reg}\subset\frak{sp}(V)^{\le r}$. Therefore is sufficient to prove that any element $X$ of rank at most $r$ lies in $\overline{\frak{sp}(V)^r_{reg}}$. Fix $X\in\frak{sp}(V)^{\le r}$. There exists an $X$-stable orthogonal decomposition $V_s\oplus V_n=V$ such that $X|_{V_n}$ is nilpotent and $X|_{V_s}$ is given (in an appropriate basis) by a matrix of type
\begin{center}$A:=\left(\begin{array}{cc}F&0\\0&-F^t\\\end{array}\right).$\end{center}
Set $r_s:=\rk X|_{V_s}$, $r_n:=\rk X|_{V_n}$. Then $r_s$ and $d_{V_s}:=\dim V_s$ are even, and $r_s+r_n=r$.

To prove that $X\in\overline{\frak{sp}(V)^r_{reg}}$ it suffices to
prove that $X|_{V_s}\in\frak{sp}(V)^{r_s}_{reg}$ and
$X|_{V_n}\in\overline{\frak{sp}(V_n)^{r_n}_{reg}}$. First of all we
check that $X|_{V_s}\in\overline{\frak{sp}(V_s)^{r_s}_{reg}}$. The
matrices of type $A$ form a subalgebra isomorphic to $\frak{gl}_m$
for $m=\frac{d_{V_s}}2$. As
$X|_{V_s}\in\frak{gl}_m^{\le\frac{r_s}2}$, we have
$X|_{V_s}\in\overline{(\frak{gl}_m)_{ss}^{\frac{r_s}2}}$. Therefore
$X|_{V_s}\in\overline{\frak{sp}(V_s)^{r_s}_{reg}}$.

Set $X':=X|_{V_n}$. It remains to show that $X'\in\overline{\frak{sp}(V_n)^{r_n}_{reg}}$ (note that $\rk X'=r_n$).

Let $X_{max}\in\frak{sp}(V_n)$ be a nilpotent element of rank $r_n$
with a single non-zero Jordan block of size $(r_n+1)\times(r_n+1)$.
Such an element of $\frak{sp}(V_n)$ exists by~\cite{CM}, and
moreover, $X'\in\overline{\EuScript O[X_{max}]}$.

It suffices to prove that
$X_{max}\in\overline{\frak{sp}(V_n)^{r_n}_{reg}}$. Using the
explicit description of the coadjoint orbits of a simplectic group
given in~\cite{CM} one can check that there exists an
$X_{max}$-stable orthogonal decomposition $V_n'\oplus V_n''=V$ such
that\begin{center}$X_{max}|_{V_n''}=0$, $d_{V_n'}=\rk
X_{max}+1$.\end{center} By Lemma~\ref{Lrkpre},
$X_{max}|_{V_n'}\in\frak{sp}(V_n')^{r_n}_{reg}$, and thus
$X_{max}\in\overline{\frak{sp}(V_n)^{r_n}_{reg}}$.\end{proof}
\begin{lemma}\label{Lrkso}\label{Lrkspe}  Let $r$ be an even integer such that $2\le r<d_V$. Then \begin{center}$\frak s(V)^{\le r}=\overline{\frak s(V)^r_{reg}}$.\end{center}\end{lemma}
\begin{proof} By definition, $\frak s(V)^r_{reg}\subset\frak s(V)^{\le r}$. Therefore its suffices to prove that any element $X$ of rank at most $r$ lies in $\overline{\frak s(V)^r_{reg}}$.
For $\frak s=\frak{so}$ the arguments are the same as in the proof
of Lemma~\ref{Lrkspo}. This applies also to the case $\frak
s=\frak{sp}$ if $\rk X$ is odd. Thus in the remainder of the
proof we can assume that $\frak s=\frak{sp}$ and $\rk X$ is even.

Let $X_{max}\in\frak{sp}(V_n)$ be a nilpotent element of rank $r_n$
with two non-zero Jordan block of sizes $r_n\times r_n$ and $2\times
2$. Such an element of $\frak{sp}(V_n)$ exists by~\cite{CM}, and
moreover, $X'\in\overline{\EuScript O[X_{max}]}$.

It suffices to prove that
$X_{max}\in\overline{\frak{sp}(V)^r_{reg}}$. There exists an
$X_{max}$-stable orthogonal decomposition $V_n'\oplus V_n''\oplus
V_n'''=V$ such that \begin{center}$d_{V_n''}=2$, $d_{V_n'}=\rk
X_{max}$, and $X|_{V_n'''}=0$.\end{center} Since
$X|_{V_n'}\in\overline{\frak{sp}(V_n')^{d_{V_n'}-1}_{reg}}$,
$X|_{V_n'\oplus V_n''}$ lies in the closure in $\frak{sp}(V_n'\oplus
V_n'')$ of the set of elements $Y\in\frak{sp}(V_n'\oplus V_n'')^{\le
r}$ such that\begin{center}$\rk Y_{ss}=\rk X_{max}-2$, $\rk
Y_{nn}=2$.\end{center} As $Y_{nn}\in\overline{\frak{sp}(\mathbb
F^4)^2_{reg}}$, we have
$X_{max}\in\overline{\frak{sp}(V)^r_{reg}}$.\end{proof}
\begin{proof}[Proof of Lemma~\ref{Lsl2} for $\frak s=\frak{so}, \frak{sp}$]By
Lemma~\ref{Leig0}, we have $\lambda=0$, and therefore $\rk X=r$. By
Lemma~\ref{Lrr} we can assume that $X$ is rank-reduced. Note that
the standard $\frak{gl}_2$-subalgebra of $\frak s_4\cong\frak{so}_4,
\frak{sp}_4$ given by matrices
\begin{center}$\left(\begin{array}{cc}F&0\\0&-F^t\\\end{array}\right),\hspace{10pt} F\in\frak{gl}_2$\end{center}
in an appropriate basis.

Assume that $\rk X$ is even. Then $X$ is conjugate (via S$(V)$) to a
direct sum
\begin{center}$\bigoplus\limits_{1\le
i\le\mathrm{rk}X_s}A(t_i)\bigoplus\limits_{1\le
j\le\mathrm{rk}X_{nn}}B\bigoplus C$\end{center}where $C$ is the
zero-operator, and $A(t)$, $B$ are of the form
\begin{center}$A(t):=\left(\begin{array}{cccc}t&0&0&0\\0&0&0&0\\0&0&-t&0\\0&0&0&0\\\end{array}\right),\hspace{10pt}B:=\left(\begin{array}{cccc}0&1&0&0\\0&0&0&0\\0&0&0&0\\0&0&-1&0\\\end{array}\right)
$.\end{center}Let $a_1[i], a_2[i], a_3[i], a_4[i]$ and $b_1[j], b_2[j], b_3[j],
b_4[j]$ denote the elements of the basis of the 4-dimensional
subspaces of $V$ corresponding to $A(t_i)$  and $B$ respectively, and let $c[k]$ be the elements of the basis of the subspace
corresponding to $C$. The restriction of the fixed form to
\begin{center}$E:=$span$\{a_1[i], a_3[i], b_1[j], b_3[i],
c[k]\mid$\\$ 1\le i\le\mathrm{rk}X_s, 1\le j\le\mathrm{rk}X_{nn},
1\le k\le d_W-2r$\}\end{center} is nondegenerate and therefore we can assume that $W=E$.
The orthogonal complement $W^\bot$ to $W$ is
the span of the remaining elements in the above basis of $V$. Then
$$\phi(\bigoplus_i\EuScript O[A(t_i)]\bigoplus_j\EuScript
O[B]\bigoplus C)\subset\bigoplus\limits_{1\le
i\le\mathrm{rk}X_s}\frak{gl}_2\bigoplus\limits_{1\le
j\le\mathrm{rk}X_{nn}}\frak{gl}_2\subset\frak s(V),$$ and we claim
that $\phi(\bigoplus_i\EuScript O[A(t_i)]\bigoplus_j\EuScript
O[B]\bigoplus C)$ contains all diagonal matrices of the form
\begin{center}diag$(\lambda_1, -\lambda_1, \lambda_2, -\lambda_2...,
\lambda_{\mathrm{rk}X}, -\lambda_{\mathrm{rk}X},0,...,
0).$\end{center} This follows from the obvious statement that
matrices in $\EuScript O[A(t)] $ for $t\ne 0$ and in $\EuScript
O[B]$ can have arbitrary values in their left top 2$\times$2-corner
for $\frak s=\frak{sp}$ and arbitrary antisymmetric values for
$\frak s=\frak{so}$.

Therefore $\phi(\EuScript O[X])$ contains all
semisimple elements of rank $r=\rk X$. As such elements are dense in
$\frak s(W)^{\le r}$ for $r$ even, $\phi(\EuScript O[X])$ is
dense in $\frak s(W)^{\le r}$.

It remains to consider the case when $\rk X$ is odd. Here $\frak
s=\frak{sp}$. Arguments similar to those above show that
$\phi(\EuScript O[X])$ contains all elements of $\frak{sp}(W)$
which has the form
\begin{center}diag$(\lambda_1, -\lambda_1, \lambda_2, -\lambda_2,...,
\lambda_{\frac{r-1}{2}}, -\lambda_{\frac{r-1}{2}}, 0, ... ,0
)\oplus\left(\begin{array}{cc}0&1\\0&0\\\end{array}\right)$\end{center}in
an appropriate basis. Then $\phi(\EuScript O[X])$ is dense in
$\frak{sp}(W)^{\le r}$ by Lemma~\ref{Lrkspo}.\end{proof}
\section{Proof of Theorems~\ref{Tlinf} and~\ref{Tlinfrk}}\label{Ssch}
In the rest of the paper we assume that $\frak g_\infty$ is a
locally simple Lie algebra which may be finitary (i.e. isomorphic to
$\frak{sl}_\infty, \frak{so}_\infty, \frak{sp}_\infty$), diagonal,
or non-diagonal. Let $\frak g_\infty=\lim\limits_{\to}\frak g_n$ for
a fixed sequence of embeddings (1). In the special case of $\frak
g_\infty\cong\frak{sl}_\infty, \frak{so}_\infty, \frak{sp}_\infty$
we assume in addition that the simple Lie algebras $\frak g_n$
satisfy the assumptions from Section~\ref{Spintro}. Let $G_n$ be the
adjoint group of $\frak g_n$. Recall that, if $J\subset${\bf
S}$^\cdot(\frak g)$ is a Poisson ideal of locally infinite
codimension, the intersections $J_n:=J\cap${\bf S}$^\cdot(\frak
g_n)$ determine proper $G_n$-stable closed subvarieties
$\Var(J_n)\subset\frak g_n^*$ which form an inverse
system\begin{center}...$\to \Var(J_n)\to[p_n]
\Var(J_{n-1})\to...$\end{center}under the natural projections $\frak
g_n^*\to\frak g_{n-1}^*$. Moreover,
$\Var(J_{n-1})=\overline{\pr_n\Var(J_n)}$ where the closure is taken
in $\frak g_{n-1}^*$.

Let $\frak k_1$ be a simple finite-dimensional Lie algebra and
$\frak k_2\subset\frak k_1$ be a simple finite-dimensional
subalgebra. The restriction of the Cartan-Killing of $\frak k_1$ to
$\frak k_2$ is proportional to the Cartan-Killing form of $\frak
k_2$ with coefficient which we denote I$_{\frak k_2}^{\frak k_1}$.
This coefficient, known as the {\it Dynkin index}, is
multiplicative: if $\frak k_3\subset\frak k_2\subset\frak k_1$ is a
chain of inclusions, then
\begin{center}I$_{\frak k_3}^{\frak k_1}=$I$_{\frak k_3}^{\frak
k_2}$I$_{\frak k_2}^{\frak k_1}$.\end{center}The Dynkin index is always a
positive integer~\cite{Dy}. Moreover, if $\frak k_2\subset\frak k_1$ are classical simple Lie algebras of the same type and of rank at least 5, then $\I_{\frak k_2}^{\frak k_1}=1$ if and only if the natural module of $\frak k_1$ decomposes over $\frak k_2$ as a natural plus a trivial module~\cite[Proposition~2.3]{DP}.
\begin{proof}[Proof of Theorem~\ref{Tlinf}]Let $J$ be a non-zero Poisson ideal of locally infinite codimension in {\bf S}$^\cdot(\frak
g_\infty)$. Without loss of generality we may assume that $J$ is a
radical ideal, as the radical of a Poisson of locally infinite
codimension ideal in ${\bf S}^\cdot(\frak g_\infty)$ is again
Poisson and of locally infinite codimension.

Fix $n$ so that $J_n=J\cap${\bf S}$^\cdot(\frak g_n)$ is non-zero and
of infinite codimension in {\bf S}$^\cdot(\frak g_n)$. The image of
any $G_{n+m}$-orbit in $\Var(J_{n+m})$ under the morphism
$\Var(J_{n+m})\to\frak g_n^*$ is not dense in $\frak g_n^*$ as it
lies in the proper closed subvariety $\Var(J_n)\subset\frak g_n^*$.
Therefore Proposition~\ref{Csl} implies that dim$(\frak g_n\cdot
V_{n+m})$ is bounded as a function on $m$. Hence the number of
non-trivial simple $\frak g_n$-constituents in $V_{n+m}$ and their
dimensions are simultaneously bounded. As a consequence, I$_{\frak
g_n}^{\frak g_{m+n}}$ is bounded as a function of $m>0$. Therefore
there exists $N>0$ such that I$_{\frak g_n}^{\frak g_{n+1}}=1$ for
all $n>N$. Now~\cite[Corollary~2.4]{DP} implies $\frak
g_\infty\cong\frak{sl}_\infty, \frak{so}_\infty$ or
$\frak{sp}_\infty$. \end{proof}
\begin{proof}[Proof of Theorem~\ref{Tlinfrk}]If $J$ is of locally finite codimension in {\bf S}$^\cdot(\frak g_\infty)$, then \begin{center}$\Var(J)=0=\frak s_\infty^{\le 0}$.\end{center} Therefore, in the rest of the proof we can assume that $J$ is non-zero and of locally infinite codimension in {\bf S}$^\cdot(\frak g_\infty)$.

Fix $n$ so that $J_n=J\cap${\bf S}$^\cdot(\frak g_n)$ is non-zero
and of infinite codimension in ${\bf S}^\cdot(\frak g_n)$. Since for
any non-zero $X\in \Var(J_{m+n})$ the image in $\frak g_n^*$ of the
$G_{m+n}$-orbit $\EuScript O[X]\subset\frak g_{m+n}^*$ is not dense
in $\frak g_n^*$, Lemma~\ref{Lsl1} implies that $\Var(J_{m+n})\subset\frak
g_{m+n}^{\le r_n'}$ for some $r_n'$ which depends on $n$ only. Let
$r_n$ be the minimal such $r_n'$. By Lemma~\ref{Lsl1}, $r_n<n$.

The inequality $r_n<n$ allows us to apply Lemma~\ref{Lsl2} when $m>
n$. It implies that the image in $\frak g_n^{\le r_n}$ of $\EuScript
O[X]\subset\frak g_{m+n}^*$ is dense in $\frak g_n^{\le r_n}$ for
any $X\in \Var(J_{m+n})$ with $\rk X=r_n$. Furthermore, by
definition, $r_n\le r_{n+m}$. Lemma~\ref{Lrkmaxsp} implies that
$r_{m+n}\ge r_n$, and therefore $r_n=r_{m+n}$. Set $r=r_n$. Then
$J=~J^{\le r}$.\end{proof}
\section{Some corollaries}\label{Sil}
Corollary~\ref{Ctlinf} implies that if $\frak g_\infty$ is a locally simple Lie algebra which is not
finitary, then any ideal in $\U(\frak g_\infty)$ is of locally finite
codimension. Furthermore, a result of A. Zhilinskii~\cite{Zh1}
claims that $\U(\frak g_\infty)$ admits an ideal of locally finite
codimension which is not the augmentation ideal if and only if $\frak
g_\infty$ is diagonal. These two statements yield the following corollary.
\begin{corollary}[Baranov's Conjecture] Let $\frak g_\infty$ be any locally simple Lie algebra. Then the augmentation ideal is the only non-zero ideal in $\U(\frak g_\infty)$ if and only if $\frak g_\infty$ is not diagonal.\end{corollary}
Furthermore, Theorems~\ref{Tlinf},~\ref{Tlinfrk} imply the
following.
\begin{corollary}Let $\frak g_\infty$ be any locally simple Lie
algebra and $I\subset\U(\frak g_\infty)$ be an ideal. Then\\
a) $\Var(I)\ne 0$ implies that $\frak g_\infty\cong\frak{sl}_\infty,
\frak{so}_\infty, \frak{sp}_\infty$;\\
b) if $\frak g_\infty=\frak{sl}_\infty, \frak{so}_\infty,
\frak{sp}_\infty$ and $I\subset\U(\frak g_\infty)$ is an ideal of
locally infinite codimension, then $\Var(I)=\frak s_\infty^{\le r}$
for some $r\in\mathbb Z_{\ge 1}$.\end{corollary}
A.~Zhilinskii~\cite{Zh3} (see also~\cite{Zh2}) has given a
description of all ideals of locally finite codimension in
 $\U(\frak g_\infty)$ for an arbitrary locally simple Lie algebra $\frak g_\infty$.
Therefore the problem of describing all ideals in $\U(\frak
g_\infty)$ gets reduced to the problem of describing all ideals of
locally infinite codimension in $\U(\frak g_\infty)$ for $\frak
g_\infty=\frak{sl}_\infty, \frak{so}_\infty, \frak{sp}_\infty$. In
the following two sections we will show in particular that all
proj-varieties $\frak s^{\le r}$ arise as associated ``varieties''
of ideals respectively of $\U(\frak{sl}_\infty),
\U(\frak{so}_\infty), \U(\frak{sp}_\infty)$.

\section{Coherent local systems of modules and a classification of prime integrable ideals of $\U(\frak g_\infty)$}\label{Sls}
In this section we review some published and unpublished results of
A. Zhilinskii and draw corollaries.
\begin{definition}An ideal $I\subset\U(\frak g_\infty)$ is {\it integrable} if for
any finitely generated subalgebra $U'\subset\U(\frak g_\infty)$, the
ideal $I\cap U'$ in $U'$ is an intersection of ideals of finite
codimension of $U'$.\end{definition} If a $\frak g_\infty$-module
$M$ is {\it integrable}, i.e. $\dim\U(\frak g')m<\infty$ for any
$m\in M$ and any finite-dimensional subalgebra $\frak g'\subset\frak
g_\infty$, the annihilator of $M$ in $\U(\frak g_\infty)$ is an
integrable ideal. Note that an equivalent definition of an
integrable $\U(\frak g_\infty)$ module is a left $\U(\frak
g_\infty)$-module $M$ for which $\dim(U'm)<\infty$ for any $m\in M$
and any finitely generated subalgebra $U'\subset \U(\frak
g_\infty)$. Integrable ideals in $\U(\frak g_\infty)$ are described
as annihilators of coherent local systems of finite-dimensional
$\frak g_n$-modules as introduced by A.~Zhilinskii in~\cite{Zh1}. We
discuss this topic below.

\subsection{Integrable ideals and coherent local systems}\label{SScls}
\begin{definition}{\it A coherent local system of modules} (further shortened as {\it c.l.s.}) for $\frak g_\infty=\varinjlim\frak g_n$ is a collection of sets \begin{center}$\{Q_n\}_{n\in\mathbb Z_{\ge1}}\subset\prod_{n\in\mathbb Z_{\ge1}}\Irr\frak g_n$\end{center} such that $Q_m=\langle Q_n\rangle_m$ for any $n>m$.\end{definition}
If $Q$ is a c.l.s., then $\cap_{z\in Q_m}\Ann_{\U(\frak
g_m)}z\subset\cap_{z\in Q_n}\Ann_{\U(\frak g_n)}z$ for any $n>m$.
Therefore $\cup_m(\cap_{z\in Q_m}\Ann_{\U(\frak g_m)}z$) is an ideal
of $\U(\frak g)$; we denote it by $I(Q)$. Note that $I(Q)$ is
integrable.

A.~Zhilinskii~\cite{Zh2},~\cite{Zh1},~\cite{Zh3} has classified
c.l.s. for any locally simple Lie algebra $\frak g_\infty$. Below we
show how this classification leads to a description of integrable
ideals of $\U(\frak g_\infty)$ for $\frak g_\infty=\frak{sl}_\infty,
\frak{so}_\infty, \frak{sp}_\infty$.

A c.l.s. $Q$ is {\it irreducible} if $Q\ne Q'\cup Q''$ with
$Q'\notin Q''$ and $Q''\notin Q'$. Let $Q$ be any c.l.s. and let
$\mathcal S(Q)$ be the set of irreducible c.l.s. which are contained
in $Q$. Then $\mathcal S(Q)$ has a finite subset of elements
$Q(1),..., Q(r)$ which are maximal by inclusion, and
$Q=\cup_rQ(r)$~\cite{Zh2}; we call $Q(r)$ the {\it irreducible
components} of $Q$. This makes apparent the analogy between c.l.s.
and algebraic varieties.

Any integrable $\frak g_\infty$-module $M$ determines a c.l.s.
$Q:=\{Q_n\}_{n\in\mathbb Z_{\ge0}}$, where
\begin{center}$Q_n:=\{z\in\Irr \frak g_n\mid\Hom_{\frak g_n}(z, M)\ne 0$\}.\end{center} We denote this relation by $Q\leftarrow M$. We also recall that an integrable $\frak g_\infty$-module $M$ is {\it
locally simple} if $M=\varinjlim M_n$ for a suitable chain
$...\subset M_n\subset M_{n+1}\subset...$ of simple
finite-dimensional $\frak g_n$-submodules $M_n$ of $M$.

\begin{proposition}[{~\cite[Lemma 1.1.2]{Zh2}}]\label{Pintm}If $Q$ is an irreducible c.l.s., then $I(Q)$ is the annihilator of some locally simple integrable $\frak g_\infty$-module. In particular, $I(Q$) is primitive and hence prime.\end{proposition}
Fix $n$. The set Irr$\frak g_n$ is parametrized by the lattice
$\Lambda_i$ of integral dominant weights of $\frak g_n$. Let $z_1,
z_2$ be isomorphism classes of simple $\frak g_n$-modules with
respective highest weights $\lambda_1, \lambda_2$. We denote by
$z_1z_2$ the isomorphism class of a simple module with highest
weight $\lambda_1+\lambda_2$. If $S_1, S_2\subset$Irr$\frak g_n$ we
set\begin{center}$S_1S_2:=\{z\in$Irr$\frak g_n\mid z=z_1z_2$ for
some $z_1\in S_1$ and $z_2\in S_2\}$.\end{center}If $Q'$, $Q''$ are
c.l.s., we denote by $Q'Q''$ the smallest c.l.s. such that
$(Q')_i(Q'')_i\subset(Q'Q'')_i$. By definition, $Q'Q''$ is the {\it
product} of $Q'$ and $Q''$. If $\frak g_\infty=\frak{sl}_\infty,
\frak{sp}_\infty$, then
by~\cite{Zh2}\begin{center}$(Q')_n(Q'')_n=(Q'Q'')_n$.\end{center}
\subsection{Zhilinskii's classification of c.l.s.} In this subsection we reproduce A.~Zhilinskii's classification of irreducible c.l.s. for $\frak g_\infty\cong\frak{sl}_\infty, \frak{so}_\infty, \frak{sp}_\infty$.

In the rest of the paper $\frak g_\infty=\frak{sl}_\infty, \frak{so}_\infty$ or $\frak{sp}_\infty$ and $\frak g_n=\frak s(V_n)$ is a sequence (1) of finite-dimensional simple Lie algebras satisfying the assumptions of Section~\ref{Spintro}. We set $V_\infty:=\varinjlim V_n$ and $(V_\infty)_*:=\varinjlim V_n^*$.

The following irreducible c.l.s. are by definition the {\it basic
c.l.s.}:
\begin{center}for $\frak g_\infty=\frak{sl}_\infty: \mathcal E\leftarrow\Lambda^\cdot V_\infty,\hspace{10pt}\mathcal L_p\leftarrow\Lambda^p V_\infty,\hspace{10pt} \mathcal
L_p^\infty\leftarrow${\bf S}$^\cdot(V_\infty\otimes\mathbb
F^p),$\\$\mathcal
R_q\leftarrow\Lambda^q(V_\infty)_*,\hspace{10pt} \mathcal R_q^\infty \leftarrow${\bf
S}$^\cdot((V_\infty)_*\otimes\mathbb F^q),
\hspace{10pt}\mathcal E^\infty$ (all modules);\\for $\frak
g_\infty=\frak{sp}_\infty: \mathcal E\leftarrow\Lambda^\cdot
V_\infty,\hspace{10pt}\mathcal L_p\leftarrow\Lambda^p
V_\infty,\hspace{10pt} \mathcal L_p^\infty\leftarrow${\bf S}$^\cdot(V_\infty\otimes\mathbb
F^p)$,\hspace{10pt}\\$\mathcal E^\infty$ (all modules);\\for $\frak g_\infty=\frak{so}_\infty: \mathcal E\leftarrow\Lambda^\cdot V_\infty,\hspace{10pt}\mathcal L_p\leftarrow\Lambda^p V_\infty,\hspace{10pt}
\mathcal L_p^\infty\leftarrow${\bf S}$^\cdot(V_\infty\otimes\mathbb F^p)$,\\
$\mathcal R$ (spinor modules),\hspace{10pt}$\mathcal E^\infty$ (all
modules),\end{center} where $p, q\in\mathbb Z_{\ge1}$.

\begin{proposition}[Unique factorization property~\cite{Zh2}]Any irreducible c.l.s. can be expressed uniquely as a product as follows:
$$\begin{tabular}{ccc}$(\mathcal L_v^\infty\mathcal
L_{v+1}^{x_{v+1}}\mathcal L_{v+2}^{x_{v+2}}...\mathcal
L_n^{x_n})~~\mathcal E^m~~(\mathcal R_w^\infty\mathcal
R_{w+1}^{z_{w+1}}\mathcal R_{w+2}^{z_{w+2}}...\mathcal R_n^{z_n})$&
for $\frak g_\infty=\frak{sl}_\infty$,&(4)\\$(\mathcal
L_v^\infty\mathcal L_{v+1}^{x_{v+1}}\mathcal
L_{v+2}^{x_{v+2}}...\mathcal L_n^{x_n})~~\mathcal E^m$ or $(\mathcal
L_v^\infty\mathcal L_{v+1}^{x_{v+1}}\mathcal
L_{v+2}^{x_{v+2}}...\mathcal L_n^{x_n})~~\mathcal E^m~~\mathcal R$&
for $\frak g_\infty=\frak{so}_\infty$,&(5)\\$(\mathcal
L_v^\infty\mathcal L_{v+1}^{x_{v+1}}\mathcal
L_{v+2}^{x_{v+2}}...\mathcal L_n^{x_n})~~\mathcal E^m$& for $\frak
g_\infty=\frak{sp}_\infty$,&(6)\end{tabular}$$where
\begin{center}$m,n, v, w\in\mathbb Z_{\ge0}$, $v, w\le n$,\\$x_i, z_j\in\mathbb Z_{\ge0}$ for $v+1\le
i\le n$ and $w+1\le j\le n$.\end{center}Here, for $v=0$, $\mathcal
L_v^\infty$ is assumed to be the identity (the c.l.s. consisting of
the isomorphism class of the trivial 1-dimensional module at all
levels), and for $w=0$, $\mathcal R_w^\infty$ is assumed to be the
identity.
\end{proposition}
We say that an irreducible c.l.s. of $\frak{so}_\infty$ is of {\it
integer type} if its expression does not contain $\mathcal R$, and
of {\it semiinteger type} otherwise. Furthermore, we define a c.l.s.
$Q$ to be {\it of finite type} if the set $Q_n$ is finite for all
$n\ge1$. It is easy to see that the ideal $I(Q)$ is of locally
finite-codimension in $\U(\frak g_\infty)$ if and only if $Q$ is of
finite type.

If $\frak g_\infty=\frak{sl}_\infty, \frak{sp}_\infty$, the
irreducible c.l.s. of finite type form a free lattice (by means of
product) generated by $\mathcal E, \mathcal L_p, \mathcal R_q$ for
$\frak g_\infty=\frak{sl}_\infty$ and by $\mathcal E, \mathcal L_p$
for $\frak g_\infty=\frak{sp}_\infty$. For $\frak
g=\frak{so}_\infty$ the set of irreducible c.l.s. of finite type
equals the union $\EuScript N\sqcup\EuScript N\mathcal R$, where
$\EuScript N$ is a free lattice generated by $\mathcal E, \mathcal
L_p$, and $\EuScript N\mathcal R:=\{N\mathcal R\mid N\in\EuScript
N\}$~\cite{Zh2}.

\subsection{Partial order by inclusion on c.l.s.}\label{SSOrd}To a c.l.s. $Q$ for $\frak{sl}_\infty$ in the form (4) A.~Zhilinskii assigns the following two non-increasing sequences of elements of
$\mathbb
Z_{\ge0}\cup\{+\infty\}$\begin{center}$\{l_i:=m+\Sigma_{j\ge
i}x_j\}_i$ and $\{r_i=m+\Sigma_{j\ge i}z_j\}_i$, \end{center}where
it is assumed $l_1=l_2=...=l_v:=+\infty=:r_1=...=r_w$. Note that
$$\lim\limits_{i\to\infty} l_i=\lim\limits_{i\to\infty}r_i=m.$$
Similarly, to a c.l.s. for $\frak{sp}_\infty$ or $\frak{so}_\infty$
in the form (5) or (6) A.~Zhilinskii assigns the non-increasing
sequence\begin{center}$\{l_i:=m+\Sigma_{j\ge
i}x_i\}_i$,\end{center}where $l_1=...=l_v:=+\infty$. Again
$\lim\limits_{i\to\infty}l_i=m$.

Zhilinskii establishes the following inclusion criterion~\cite{Zh2}.
A c.l.s. $Q$ for $\frak{sl}_\infty$ contains a c.l.s. $Q'$ if and
only if, for some $a, b\in\mathbb Z_{\ge0}$, we have
\begin{center}$a+b=m-m'$, $l_i\ge l_i'+a$, $r_i\ge r_i'+b$\end{center} for
the corresponding sequences $\{l_i\}, \{r_i\}$ and $\{l_i'\},
\{r_i'\}$. If $Q$ and $Q'$ are c.l.s. of $\frak{sp}_\infty$, then
$Q'\subset Q$ if and only if\begin{center}$l_i\ge
l_i'$.\end{center} Finally, if $Q$ and $Q'$ are c.l.s. of $\frak{so}_\infty$, then
$Q'\subset Q$ if and only if $Q$ and $Q'$ have the same integer
or semiinteger type, and
\begin{center}$l_i\ge l_i'$.\end{center}
\begin{corollary}\label{Lmid} If $\frak g_\infty=\frak{sl}_\infty, \frak{sp}_\infty$, the only minimal c.l.s. is the trivial c.l.s. (i.e. the c.l.s. $Q$ with $Q_n$ being the isomorphism class of the 1-dimensional trivial $\frak g_n$-module for all $n\ge1$). If $\frak g_\infty=\frak{so}_\infty$, there are two minimal c.l.s.: the trivial one and $\mathcal R$.\end{corollary}
\begin{proof}Follows from the inclusion criterion of A.~Zhilinskii.\end{proof}
\subsection{Tensor product and ideals}\label{SStpi}Fix $i$. If $S_1, S_2\subset\Irr\frak g_n$ we
set\begin{center}$S_1\otimes S_2:=\{z\in\Irr\frak g_n \mid$Hom$(z,
z_1\otimes z_2)\ne 0$ for some $z_1\in S_1$ and $z_2\in
S_2\}$.\end{center}Furthermore, it is clear that the tensor product
of two c.l.s. is a well defined c.l.s.. One can check
that\begin{center}$(\mathcal L_v^\infty\mathcal
L_{v+1}^{x_{v+1}}\mathcal L_{v+2}^{x_{v+2}}...\mathcal
L_n^{x_n})~~\mathcal E^m~~(\mathcal R_w^\infty\mathcal
R_{w+1}^{z_{w+1}}\mathcal R_{w+2}^{z_{w+2}}...\mathcal
R_n^{z_n})=$\\$(\mathcal L_1^\infty)^{\otimes v}\otimes(\mathcal
R_1^\infty)^{\otimes w}\otimes((\mathcal L_1^{x_{v+1}}\mathcal
L_2^{x_{v+2}}...\mathcal L_{n-v}^{x_n})~~\mathcal E^m~~(\mathcal
R_1^{z_{w+1}}\mathcal R_2^{z_{w+2}}...\mathcal
R_{n-w}^{z_n}))$\end{center}for $\frak g_\infty=\frak{sl}_\infty$,
and\begin{center}$(\mathcal L_v^\infty\mathcal
L_{v+1}^{x_{v+1}}\mathcal L_{v+2}^{x_{v+2}}...\mathcal
L_n^{x_n})~~\mathcal E^m=(\mathcal L_1^\infty)^{\otimes
v}\otimes((\mathcal L_1^{x_{v+1}}\mathcal L_2^{x_{v+2}}...\mathcal
L_{n-v}^{x_n})~~\mathcal E^m)$,\\\hspace{12pt}$(\mathcal
L_v^\infty\mathcal L_{v+1}^{x_{v+1}}\mathcal
L_{v+2}^{x_{v+2}}...\mathcal L_n^{x_n})~~\mathcal E^m~~\mathcal
R=(\mathcal L_1^\infty)^{\otimes v}\otimes((\mathcal
L_1^{x_{v+1}}\mathcal L_2^{x_{v+2}}...\mathcal
L_{n-v}^{x_n})~~\mathcal E^m~~\mathcal R)$\hspace{12pt}
(7)\end{center}for $\frak g_\infty=\frak{so}_\infty,
\frak{sp}_\infty$, where $(\mathcal L_1^\infty)^{\otimes
v}:=\mathcal L_1^\infty\otimes\mathcal
L_1^\infty\otimes...\otimes\mathcal L_1^\infty (v$ times). Note that
formula (7) applies to $\frak{so}_\infty$-case only.

The above formulas yield a different parametrization of the
irreducible c.l.s.: for $\frak g_\infty=\frak{sl}_\infty$ the
irreducible c.l.s. are parametrized by triples $(v, w, Q_f)$, where
$v, w\in\mathbb Z_{\ge0}$ and $Q_f$ is an irreducible c.l.s. of
finite type, and for $\frak g_\infty=\frak{so}_\infty,
\frak{sp}_\infty$ the irreducible c.l.s. are parametrized by pairs
$(v, Q_f)$, where $v\in\mathbb Z_{\ge0}$ and $Q_f$ is an irreducible
c.l.s. of finite type.

For $\frak g_\infty=\frak{sl}_\infty$ we set
\begin{center}$\cls(v, w, Q_f):=(\mathcal L_1^\infty)^{\otimes
v}\otimes(\mathcal R_1^\infty)^{\otimes w}\otimes Q_f$\end{center}
and
\begin{center}$I(v, w, Q_f):=I(\cls(v, w, Q_f))$.\end{center}
However, it is easy to check that the annihilators of $\mathcal
L_1^\infty$ and $\mathcal R_1^\infty$ in $\U(\frak{sl}_\infty)$
coincide. Therefore
\begin{center}$I(v, w, Q_f)$=$I(v+w, 0, Q_f)$.\end{center}
In what follows we call an irreducible c.l.s. for $\frak
g_\infty=\frak{sl}_\infty$ of the form $\cls(v, 0, Q_f)$ a {\it left
irreducible} c.l.s. An arbitrary {\it left} c.l.s. is defined as a
finite union of left irreducible c.l.s. We denote the left
irreducible c.l.s. for $\frak g_\infty=\frak{sl}_\infty$ by $\cls(v,
Q_f)$. We also set $I(v, Q_f):=I(v, 0, Q_f)$.

Note that for $\frak g_\infty=\frak{so}_\infty, \frak{sp}_\infty$ an
arbitrary irreducible c.l.s. corresponds to a pair $(v, Q_f)$.
Therefore the notations $\cls(v, Q_f)$ and $I(v, Q_f)$ make sense in
all three cases: $\frak{sl}_\infty, \frak{so}_\infty,
\frak{sp}_\infty$.
\subsection{Classification of integrable ideals in $\U(\frak g_\infty)$}\label{SSsl}Recall that $\frak g_n=\frak{sl}_{n+1}$ for $\frak g_\infty\cong\frak{sl}_\infty$. In this case the space of weights of $\frak g_n$ is identified
with the set of $(n+1)$-tuples $(\lambda_1,..., \lambda_{n+1})$, $\lambda_i\in\mathbb F$, up to the equivalence relation \begin{center}$(\lambda_1,..., \lambda_{n+1})\sim(\lambda_1+k,..., \lambda_{n+1}+k)$\end{center}for $k\in\mathbb F$. The lattice of integral dominant weights is identified with the set of weights such that $\lambda_{i+1}-\lambda_i\in\mathbb Z_{\ge0}$ for $1\le i\le n-1$.

\label{SSsop} For $\frak g_\infty\cong\frak{sp}_\infty$, the space
of weights of $\frak g_n\cong\frak{sp}_{2n}$ is identified with the
set of $n$-tuples $(\lambda_1,..., \lambda_n)$, $\lambda_i\in\mathbb
F$. The lattice of integral dominant weights is identified with the
set of weights such that $\lambda_i\in\mathbb Z_{\ge0}$ and
$\lambda_{i+1}\ge\lambda_i$ for~$1\le i\le n-1$.

For $\frak g_\infty\cong\frak{so}_\infty$, we assume that $\frak
g_n\cong\frak{so}_{2n}$. The space weights of $\frak
g_n=\frak{so}_{2n}$ is identified with the set of $n$-tuples
$(\lambda_1,..., \lambda_n)$, $\lambda_i\in\mathbb F$. The lattice
of integral dominant weights is identified with the set of weights
such that $\lambda_i\in\frac{1}{2}\mathbb Z$,
$\lambda_{i+1}-\lambda_i\in\mathbb Z_{\le0}$ for~$1\le i\le n-1$ and
$\lambda_{n-1}\ge|\lambda_n|$.

Let $\mathcal W_n$ denote the Weyl group of $\frak g_n$, and let
$\rho_n$ be the half-sum of positive roots of $\frak g_n$. The set
of radical ideals of Z$_{\U(\frak g_n)}$ is identified with the set
of Zariski-closed $\mathcal W_n$-stable subsets of the space of
weights of $\frak g_n$.

Let $Q$ be a c.l.s. For any $n\in\mathbb Z_{\ge1}$ we denote by
$\overline{Q_n}$ the Zariski-closure of the set of highest weights
of the isomorphism classes of simple $\frak g_n$-modules from $Q_n$
in the space of weights of $\frak g_n$.

\begin{lemma}\label{Cirs}\label{Csp2012}\label{Csp2011}a) For any $v\in\mathbb Z_{\ge0}$ and any c.l.s. $Q_f$ of finite type, the following conditions are equivalent:
\begin{enumerate}
\setcounter{AP}{1}\item[(\roman{AP})] $(v, Q_f)=(v', Q_f')$,
\setcounter{AP}{2}\item[(\roman{AP})] $\mathcal W_n(\rho_n+\overline{\mathrm{cls}(v, Q_f)_n})=\mathcal W_n(\rho_n+\overline{\mathrm{cls}(v', Q_f')_n)}$ for all $n\in\mathbb Z_{\ge1}$.
\end{enumerate}
b) Let $z\in\Irr \frak g_n$ and let $\lambda_z$ be the highest
weight of a representative of $z$. For $\frak
g_\infty\cong\frak{so}_\infty$, $\frak{sp}_\infty$, we
have\begin{center}$\lambda_z+\rho_n\in\mathcal
W_n(\rho_n+\overline{\mathrm{cls}(v, Q_f)_n})$\end{center} if and
only if $z\in\cls(v, Q_f)_n$. For $\frak
g_\infty\cong\frak{sl}_\infty$, we
have\begin{center}$\lambda_z+\rho_n\in\mathcal
W_n(\rho_n+\overline{\mathrm{cls}(v, Q_f)_n})$\end{center}if and
only if there exist $v', v''\in\mathbb Z_{\ge0}$ such that
$v'+v''=v$ and $z\in\cls(v', v'', Q_f)_n$.\\
c) Let $v\in\mathbb Z_{\ge0}$ and $Q_f$ be a c.l.s. for $\frak
g_\infty$. Then, for $\frak g_\infty\cong\frak{so}_\infty$,
$\frak{sp}_\infty$ we have $Q(I(v, Q_f))=\cls(v, Q_f)$. For $\frak
g_\infty\cong\frak{sl}_\infty$ we have \begin{center}$Q(I(v,
Q_f))=\cup_{v'+v''=v}\cls(v', v'', Q_f)$.\end{center}
\end{lemma}
\begin{proof}Part a) follows from the following explicit description
of the closure $\overline{\mathrm{cls}(v, Q_f)_{n+v}}$: if $Q_f$ is
an irreducible c.l.s. of finite type, then $(\lambda_1,...,
\lambda_{n+v})\in\overline{\mathrm{cls}(v, Q_f)_{n+v}}$ if and only
if the simple $\frak g_n$-module with highest weight
$(\lambda_{v+1},..., \lambda_{n+v})$ lies in $(Q_f)_n$.

Part b) is a straightforward corollary of part a). We proceed to c).

We fix $v, Q_f, n$. Let $z\in Q(I(v, Q_f))_n$ and $\lambda_z$ be the
highest weight of $z$. By definition, $I\cap\U(\frak
g_n)\subset\Ann_{\U(\frak g_n)}z$. Then $I\cap\mathrm Z(\U(\frak
g_n))\subset(\Ann_{\U(\frak g_n)}z)\cap\mathrm Z(\U(\frak g_n))$.
This inclusion is equivalent to the
condition\begin{center}$\lambda_z+\rho_n\in\mathcal
W_n(\rho_n+\overline{\mathrm{cls}(v, Q_f)_n})$.\end{center}
Therefore $z\in\cls(v, Q_f)_n$ for $\frak
g_\infty\cong\frak{sp}_\infty, \frak{so}_\infty$, and
$z\in\cup_{v'+v''=v}\cls(v', v'', Q_f)_n$ for $\frak
g_\infty\cong\frak{sl}_\infty$ by b).

Hence $Q(I(v, Q_f))=\cls(v, Q_f)$ for $\frak
g_\infty\cong\frak{so}_\infty, \frak{sp}_\infty$.

Assume that $\frak g_\infty\cong\frak{sl}_\infty$. Then $I(v', v'',
Q_f)=I(v'+v'', Q_f)$ and thus $\cls(v', v'', Q_f)\subset Q(I(v'+v'',
Q_f))$. Hence $Q(I(v, Q_f))\cup_{v'+v''=v}\cls(v', v'',
Q_f)$.\end{proof}

For any ideal $I\subset\U(\frak g_\infty)$,
set\begin{center}$Q(I)_n:=\{z\in\Irr{\frak g_n}\mid I\cap\U(\frak
g_n)\subset\Ann_{\U(\frak g_n)}z\}$\end{center} and note that $Q(I)$
is a well-defined c.l.s.. Note that if $I$ is integrable, then $I$
is the annihilator of $Q(I)$.
\begin{proposition}Assume $\frak g_\infty=\frak{sl}_\infty, \frak{so}_\infty, \frak{sp}_\infty$. An integrable ideal of $\U(\frak g_\infty)$ is prime if and only if it is primitive.~\footnote{The analogous statement is false for $\frak g_n$. For instance, $I(\mathcal L_1)\cap\U(\frak g_n), n\ge2,$ is an integrable prime ideal of $\U(\frak g_n)$ which is not primitive.}\end{proposition}
\begin{proof}Let $I$ be a prime integrable ideal of $\U(\frak g_\infty)$. Then $I$ is the annihilator of $Q(I)$. Let $Q_1,..., Q_s$ be the irreducible components of $Q(I)$. We have\begin{center}$I=I(Q(I))\subset\cap_{i\le s}I(Q_s)$\end{center}and\begin{center}$I(Q_1)I(Q_2)...I(Q_s)\subset I(Q(I))=I$.\end{center}
Therefore $I=I(Q(I))$ coincides with $I(Q_i)$ for some irreducible
c.l.s. $Q_i$. Thus the ideal $I(Q_i)$ is primitive by
Proposition~\ref{Pintm}. On the other hand, any primitive ideal is
prime.\end{proof}

For $\frak g_\infty=\frak{sl}_\infty$, we denote by $Q_l(I)$ the
union of all left irreducible components of $Q(I)$. For $\frak
g_\infty=\frak{so}_\infty, \frak{sp}_\infty$, we put $Q_l(I)=Q(I)$.

\begin{theorem}\label{Tclid}a) The maps
\begin{center}$I\mapsto Q_l(I),$\\$Q\mapsto I(Q)$\end{center}
are mutually inverse bijections between the set of prime integrable ideals $I\subset\U(\frak g_\infty)$ and the set of irreducible left c.l.s..\\
b) If $\frak g_\infty=\frak{so}_\infty, \frak{sp}_\infty$, the maps in a) extend to mutually inverse anti-isomorphisms\begin{center}$\hspace{150pt}I\mapsto Q(I),\hspace{150pt}(8)$\\$\hspace{150pt}Q\mapsto I(Q)\hspace{150pt}(9)$\end{center}between the lattice of integrable ideals in $\U(\frak g_\infty)$ and the lattice of c.l.s. for $\frak g_\infty$~\footnote{Contrary to our conventions from Section~\ref{Spre}, here we consider $\U(\frak g_\infty)$ as an integrable ideal and \{0\} as c.l.s.}.\\
c) If $\frak g_\infty=\frak{sl}_\infty$, any integrable ideal of
$\U(\frak{sl}_\infty)$ equals $I(Q)$ for some left c.l.s. $Q$ for
$\frak g_\infty$.\end{theorem}
\begin{proof}Let $I$ be a prime integrable ideal of $\U(\frak g_\infty)$. By definition, $I$ is the annihilator of $Q(I)$. Let $Q_1,..., Q_s$ be the irreducible components of $Q(I)$. Then \begin{center}$I=I(Q(I))\subset\cap_{i\le s}I(Q_s)$\end{center}and\begin{center}$I(Q_1)I(Q_2)...I(Q_s)\subset I(Q(I))=I$.\end{center}
Therefore $I=I(Q(I))$ coincides with $I(Q_i)$ for some irreducible
c.l.s. $Q_i$. Hence, by Subsection~\ref{SStpi}, $I=I(v, Q_f)$ for
some $v\in\mathbb Z_{\ge0}$ and some c.l.s. of finite type $Q_f$.

To prove a) it remains to check that the annihilators of the c.l.s.
corresponding to distinct pairs $(v, Q_f)$ and $(v', Q_f')$ are
distinct. For $\frak g_\infty=\frak{sl}_\infty$ we have to prove
that $I(v, 0, Q_f)=I(v', 0, Q_f')$ if and only if $v=v'$ and
$Q_f=Q_f'$. For $\frak g_\infty=\frak{so}_\infty, \frak{sp}_\infty$,
we have to check that $I(v, Q_f)=I(v', Q_f')$ if and only if $v=v'$
and $Q_f=Q_f'$. Both statements follow from Lemma~\ref{Cirs}. Hence,
a) is proved.

To prove b) we assume that $\frak g_\infty\cong\frak{so}_\infty,
\frak{sp}_\infty$. Let $Q$ be any c.l.s.. Then $Q=\cup_{i\le s}Q_i$
for some irreducible c.l.s. $Q_i$, and $I(Q)=\cap_{i\le s}I(Q_i)$.
On the other hand, $Q_l(\cap_{i\le s}I(Q_i))=Q(\cap_{i\le s}I(Q_i))$
by definition, and\begin{center}$Q(\cap_{i\le s}I(Q_i))=\cup_{i\le
s}(Q(I(Q_i)))=\cup_{i\le s}Q_i=Q$\end{center} by a). Therefore the
maps (8) and (9) are mutually inverse. In addition, it is now
obvious that both maps are anti-homomorphisms of lattices. This
proves b).

To prove c) we assume that $\frak g_\infty\cong\frak{sl}_\infty$.
Let $I$ be an integrable ideal of $\U(\frak{sl}_\infty)$. Then
$I=\cap_{i\le s}I_s$, where $I_i$ are prime integrable ideals of
$\U(\frak{sl}_\infty)$. For any $i\le s$, $I_i=I(v_i, (Q_f)_i)$ for
some $(v_i, (Q_f)_i)$ as in a). In particular, $I_i$ is the
annihilator of a left local system $Q_i:=(v_i, 0, (Q_f)_i)$. Then
$I$ is the annihilator of $\cup_{i\le s}Q_i$. This proves
c).\end{proof} Note that for $\frak g_\infty\cong\frak{sl}_\infty$
the annihilators of $\mathcal L_2^\infty$ and $(\mathcal
L_2^\infty)\cup(\mathcal L_1^\infty\mathcal R_1)$  coincide. This
shows in particular that the one-to-one correspondence between left
irreducible c.l.s. of $\frak g_\infty$ and prime integrable ideals
of $\U(\frak g_\infty)$ can not be extended to an anti-isomorphism
between the corresponding lattices.

Nevertheless,  Theorem~\ref{Tclid}~c) provides a certain description
of general integrable ideal of $\U(\frak{sl}_\infty)$: it yields a
surjection from the set of c.l.s. of $\frak{sl}_\infty$ to the set
of integrable ideals of $\U(\frak{sl}_\infty)$. Two c.l.s. $Q_1,
Q_2$ determine the same integrable ideal $I(Q_1)=I(Q_2)$ if and only
if $Q(I(Q_1))=Q(I(Q_2))$. For any irreducible c.l.s. $Q$, the c.l.s.
$Q(I(Q))$ is described by Lemma~\ref{Cirs}~c). If $Q$ is a reducible
c.l.s., i.e. $Q=\cup_{j\le s}Q_j$ for some irreducible c.l.s.
$Q_1,..., Q_j$, then $Q(I(Q))=\cup_{j\le s}Q(I(Q_j))$. This allows
in principle to check when $Q(I(Q_1))=Q(I(Q_2))$.

Furthermore, let's point out that Theorem~\ref{Tclid}, together with
A. Zhilinskii's result~\cite{Zh2},~\cite{Zh3} that any coherent
local system is the union of finitely many coherent local systems,
implies that the lattice of integrable ideals of $\U(\frak
g_\infty)$ satisfies the ascending chain condition. This has already
been stated in~\cite{Zh2} and~\cite{Zh3}. We would like to think of
this result as of ``relative N\"otherianity'' of the algebra
$\U(\frak g_\infty)$. We don't know whether $\U(\frak g_\infty)$ is
two-sided N\"otherian.

Theorem~\ref{Tclid} and Corollary~\ref{Lmid} imply immediately that
if $\frak g_\infty=\frak{sl}_\infty, \frak{sp}_\infty$, the
augmentation ideal is the unique integrable maximal ideal of
$\U(\frak g_\infty)$. For $\frak g_\infty=\frak{so}_\infty$ there
are two integrable maximal ideals: the augmentation ideal and the
``spinor ideal'' $I(\mathcal
R)$.
\begin{corollary}The algebras $\U(\frak{sl}_\infty), \U(\frak{so}_\infty$) and $\U(\frak{sp}_\infty$) are pairwise nonisomorphic.\end{corollary}
\begin{proof}The algebras
 $\U(\frak{sl}_\infty), \U(\frak{sp}_\infty)$ have each a unique integrable maximal ideal, while the algebra
 $\U(\frak{so}_\infty$) has two integrable maximal ideals. Hence
 $\U(\frak{sl}_\infty)\not\cong\U (\frak{so}_\infty)$ and $\U(\frak{sp}_\infty)\not\cong\U (\frak{so}_\infty)$ .

Consider now prime submaximal integrable ideals in
$\U(\frak{sl}_\infty)$ and $\U(\frak{sp}_\infty)$, i.e. integrable
prime ideals which are properly contained only in integrable maximal
ideals. Using the inclusion criterion of irreducible c.l.s. from
Subsection~\ref{SSOrd}, one checks immediately (using
Theorem~\ref{Tclid}~a)) that $\U(\frak{sl}_\infty)$ has two such
submaximal ideals, namely $I(\mathcal L_1)$ and $I(\mathcal R_1)$,
while $\U(\frak{sp}_\infty)$ has a single such ideal $I(\mathcal
L_1)$. Hence $\U(\frak{sl}_\infty)\not\cong\U(\frak{sp}_\infty)$.
\end{proof}
Finally, note that$$\hspace{82pt}\Var(I(v,
Q_f))=\begin{cases}\frak{sl}_\infty^{\le v}&\mathrm{for~}\frak
g_\infty=\frak{sl}_\infty,\\\frak{so}_\infty^{\le
2v}&\mathrm{for~}\frak
g_\infty=\frak{so}_\infty,\\\frak{sp}_\infty^{\le
2v}&\mathrm{for~}\frak
g_\infty=\frak{sp}_\infty.\\\end{cases}\hspace{82pt}(10)$$ Since
$\frak{so}^{\le2v+1}_\infty=\frak{so}^{\le 2v}_\infty$, all possible
associated ``varieties'' of ideals in $\U(\frak{sl}_\infty)$ and
$\U(\frak{so}_\infty)$ appear in the right-hand side of (10). The
only possible associated ``varieties'' which do not appear in the
right-hand side of (10) are $\frak{sp}^{\le 2v+1}_\infty$ for
$v\in\mathbb Z_{\ge0}$, and these proj-varieties are not associated
``varieties'' of integrable ideals of $\U(\frak{sp}_\infty)$.
\section{Some non-integrable ideals of $\U(\frak{sp}_\infty)$}
We will now provide non-integrable ideals of $\U(\frak{sp}_\infty)$
whose associated ``varieties'' are the proj-varieties
$\frak{sp}_\infty^{\le 2v+1}$ for $z\in\mathbb Z_{\ge0}$. We start
with a lemma and a corollary.

Let $\frak g$ be a semisimple Lie algebra. Consider the coproduct
\begin{center}$\Delta: \U(\frak g)\to \U(\frak
g)\otimes\U(\frak g)~(x\mapsto x\otimes1+1\otimes x$ for $x\in\frak
g)$.\hspace{20pt}\end{center} In what follows we will denote the
``diagonal'', ``left'' and ``right'' copies of $\frak g$
respectively by $\frak g_{\Delta}, \frak g_l, \frak g_r$, i.e. we
have $$\Delta: \U(\frak g_{\Delta})\subset\U(\frak
g_l)\otimes\U(\frak g_r).$$
\begin{lemma}\label{Ltens}Let $I_l, I_r$ be
ideals of $\U(\frak g)$. Then
$$\Var(I_l)+\Var(I_r)\subset\Var(\U(\frak g_\Delta)/(\U(\frak g_\Delta)\cap(I_l\otimes 1+1\otimes
I_r))),$$where $\Var(I_l)+\Var(I_r)$ is a pointwise sum of the
varieties $\Var(I_l)$ and $\Var(I_r)$ inside $\frak g^*$.\end{lemma}
\begin{proof}We have $$\gr(\U(\frak g_\Delta)\cap(I_l\otimes\U(\frak g_r)+\U(\frak g_l)\otimes
I_r))\subset\gr(I_l\otimes\U(\frak g_r)+\U(\frak g_l)\otimes
I_r)=\gr I_l\otimes{\bf\mathrm S}^\cdot(\frak g_r)+{\bf\mathrm
S}^\cdot(\frak g_l)\otimes\gr
I_r.$$Therefore$$\Var(\gr(I_l))+\Var(\gr(I_r))\subset\Var(\gr(\U(\frak
g_\Delta)\cap(I_l\otimes\U(\frak g_r)+\U(\frak g_l)\otimes
I_r))).$$\end{proof}
\begin{corollary}\label{Ctens}Let $M_l, M_r$ be $\frak g$-modules with annihilators $I_l, I_r$. Let $I_{\Delta}$ be the annihilator of $M_l\otimes M_r$. Then\begin{center}$\Var(I_l)+\Var(I_r)\subset\Var(I_{\Delta})$.\end{center}\end{corollary}
\begin{proof}We first consider $M_l\otimes M_r$ as $\U(\frak g_l)\otimes\U(\frak g_r)$-module. Then the structure of $\frak g$-module on $M_l\otimes M_r$ comes from the homomorphism $\Delta$. Therefore $$\Ann_{\U(\frak g)}(M_1\otimes M_2)=\U(\frak g_\Delta)\cap(I_l\otimes 1+1\otimes I_r).$$ Hence $\Var(I_l)+\Var(I_r)\subset\Var(I_{\Delta})$ by Lemma~\ref{Ltens}.\end{proof}
We now provide a maximal ideal $I\subset\U(\frak{sp}_\infty)$ such
that $\Var(I)=\frak{sp}(V)^{\le 1}$.

Consider the Weyl algebra $\mathbb W_\infty$ of infinitely many
variables, i.e. the associative algebra generated by $x_1, x_2,...,
x_n,...,
\partial_{x_1},
\partial_{x_2},..., \partial_{x_n},...$ with
relations\begin{center}$[x_i, x_j]=0,\hspace{10pt}[\partial_{x_i},
\partial_{x_j}]=0$ for all $i, j$;\\$[\partial_{x_i}, x_i]=1$ for $1\le i\le n-1$;\hspace{10pt}$[\partial_{x_i}, x_j]=0$ for all $i\ne j$.\end{center}
The subspace spanned by $1, x_ix_j, x_i\partial_{x_j},
\partial_{x_j}\partial_{x_i}$ is a Lie
subalgebra of the Lie algebra associated with $\mathbb W_\infty$.
This subalgebra $\tilde{\frak s}$ has a 1-dimensional center
generated by 1, and its derived  subalgebra $[\tilde{\frak s},
\tilde{\frak s}]$ is isomorphic to $\frak{sp}_\infty$. Therefore we
have a homomorphism $\U(\frak{sp}_\infty)\to \mathbb W_\infty$; the
kernel of this homomorphism is an ideal $I_{\mathbb W}$. For any
$n$, $\frak g_n=\frak{sp}_{2n}$ and $I_{\mathbb W}\cap\U (\frak
g_n)$ is a Joseph ideal~\cite{Jo}. We denote by $V_{\mathbb W}$ the
algebra $\mathbb F[x_1, x_2,...]$ considered as an
$\frak{sp}_\infty$-module.

It is well known that the ideal $I_{\mathbb W}\cap \U(\frak g_n)$ is
maximal (thus primitive and prime) in $\U(\frak{sp}_{2n})$ and not
integrable. Hence $I_{\mathbb W}$ is a maximal ideal in
$\U(\frak{sp}_\infty)$ which is not integrable. The zero-set
$\Var(I_{\mathbb W})$ coincides with $\frak{sp}_\infty^{\le 1}$
(note that any element of rank 1 in $\frak{sp}_{2n}$ is nilpotent).

More generally, we have the following.
\begin{proposition}For any $v\in\mathbb Z_{\ge0}$ we have
$$\Var(\Ann_{\U(\frak{sp}_\infty)}({\bf S}^\cdot(V_\infty\otimes\mathbb F^v)\otimes V_{\mathbb W}))=\frak{sp}_\infty^{\le 2v+1}.$$\end{proposition}
\begin{proof}By Corollary~\ref{Ctens}, we have
\begin{center}$\overline{\Var(\Ann_{\U(\frak{sp}_\infty)}({\bf S}^\cdot(V_\infty\otimes\mathbb
F^v)))+\Var(I_{\mathbb W})}=\overline{\frak{sp}_\infty^{\le
2v}+\frak{sp}_\infty^{\le 1}}$=\\=$\frak{sp}_{\infty}^{\le
2v+1}\subset\Var(\Ann_{\U(\frak{sp}_\infty)}({\bf
S}^\cdot(V_\infty\otimes\mathbb F^v)\otimes V_{\mathbb
W})).$\end{center}

For any $n$ there is a natural embedding $\frak{sl}(V_n)
\to\frak{sp}(V_n\oplus V_n^*)$, and these embeddings define an
embedding $\frak{sl}_\infty\to\frak{sp}_\infty$. We have
\begin{center}$\Ann_{\U(\frak{sp}_\infty)}({\bf S}^\cdot(V_\infty\otimes\mathbb
F^v)\otimes V_{\mathbb W})\cap\U(\frak{sl}_\infty)=I(2v+1,\mathbb
F),$\end{center}where $\mathbb F$ stands for the c.l.s.
corresponding to the augmentation ideal, and therefore the image of
$\Var(\Ann_{\U(\frak{sp}_\infty)}({\bf
S}^\cdot(V_\infty\otimes\mathbb F^v)\otimes V_{\mathbb W}))$ under
the projection $\frak{sp}_\infty^*\to\frak{sl}_{\infty}^*$ lies in
$\frak{sl}_\infty^{\le 2v+1}$. Hence, by
Lemma~\ref{Lglsp},$$\Var(\Ann_{\U(\frak{sp}_\infty)}({\bf
S}^\cdot(V_\infty\otimes\mathbb F^v)\otimes V_{\mathbb
W}))\subset\frak{sp}_\infty^{\le 2v+1}.$$This shows that
$\Var(\Ann_{\U(\frak{sp}_\infty)}({\bf
S}^\cdot(V_\infty\otimes\mathbb F^v)\otimes V_{\mathbb
W}))=\frak{sp}_\infty^{\le 2v+1}$.
\end{proof}

\end{document}